\documentclass[a4paper]{amsart}
\usepackage{amsmath}
\usepackage{hyperref}

\newcommand{\abs}[1]{\lvert #1 \rvert}
\newcommand{\labs}[1]{\left| #1 \right|}

\newcommand{\scalarh}[2]{\langle #1 \,|\, #2 \rangle}
\newcommand{\scalarl}[2]{( #1 \,|\, #2 )}
\newcommand{\scalar}[2]{( #1 , #2 )}

\newcommand{\rscalarh}[2]{\Re \langle #1 \,|\, #2 \rangle}
\newcommand{\rscalarl}[2]{\Re ( #1 \,|\, #2 )}

\newcommand{\normh}[1]{\lVert #1 \rVert}
\newcommand{\norml}[2][2]{\labs{#2}_{#1}}
\newcommand{\settc}[2]{\bigl\{\,#1 \bigm\vert #2\,\bigr\}}
\newcommand{\tow}{\rightharpoonup}

\newcommand{\malf}{\boldsymbol{\alpha}}

\DeclareMathOperator{\esp}{e}
\newcommand{\J}[1][w]{J_{#1}^{\lambda}}
\newcommand{\dJ}[2][w]{dJ_{#1}^{\lambda}(\eta)[#2]}
\newcommand{\aeta}{a(\eta)}
\newcommand{\daeta}[1]{da(\eta)[#1]}

\newcommand{\E}[1][]{\mathcal{E}_{#1\lambda}}
\newcommand{\Eu}{\mathcal{E}_{1}}

\newcommand{\etl}{\eta_{\theta\lambda}}
\newcommand{\pel}{\psi_{\theta\lambda}}

\numberwithin{equation}{section}

\newtheorem{thm}[equation]{Theorem}
\newtheorem{lem}[equation]{Lemma}
\newtheorem{prop}[equation]{Proposition}

\theoremstyle{remark}
\newtheorem{rem}[equation]{Remark}

\title{Normalized solutions for a nonlinear Dirac equation}

\author[Coti Zelati]{Vittorio Coti Zelati}

\email[Coti Zelati]{vittorio.cotizelati@unina.it}

\address[Coti Zelati]{Dipartimento di Matematica Pura e Applicata 
``R.~Caccioppoli''\\
Universit\`a di Napoli ``Federico II''\\
via Cintia, M.S.~Angelo\\
80126 Napoli (NA), Italy}

\author{Margherita Nolasco} 

\email[Nolasco]{nolasco@univaq.it}

\address[Nolasco]{Dipartimento di Ingegneria e Scienze
dell'informazione e Matematica\\
Universit\`a dell'Aquila\\
via Vetoio, Loc.~Coppito\\
67010 L'Aquila (AQ) Italia}


\thanks{Margherita Nolasco was partially supported by the Grant
MIUR-PRIN-20227HX33Z – ``Pattern formation in nonlinear phenomena''}
\begin{document}

\begin{abstract}
	We prove the existence of a normalized, stationary solution $\psi
	\colon \mathbb{R}^{3} \to \mathbb{C}^{4}$ with frequency $\omega >
	0$ of the nonlinear Dirac equation. The result covers the case in
	which the nonlinearity is the gradient of a function of the form
	\begin{equation*}
		F(\Psi) = a\abs{\scalar{\Psi}{\gamma^{0}
		\Psi}}^{\frac{\alpha}{2}} + b\abs{\scalar{\Psi}{\gamma^{1}
		\gamma^{2} \gamma^{3} \Psi}}^{\frac{\alpha}{2}}
	\end{equation*}
	with $\alpha \in (2,\frac{8}{3}]$, $b \geq 0$ and $a > 0$
	sufficiently small. Here $\gamma^{i}$, $i = 0,\ldots, 3$ are the 
	$4 \times 4$ Dirac's matrices.

	We find the solution as a critical point of a suitable functional
	restricted to the unit sphere in $L^{2}$, and $\omega$ turns out
	to be the corresponding Lagrange multiplier.
\end{abstract}

\maketitle

\section{Introduction}

The nonlinear Dirac equation is a simplified model describing
self-interacting fermions (electrons). The equation we consider is the
following:
\begin{equation}
	\label{eq:equazioneCompleta}
	(-i\gamma^{\mu} \partial_{\mu} + m) \Psi = \gamma \beta \nabla
	F(\Psi) \qquad \text{in } \mathbb{R} \times \mathbb{R}^{3}
\end{equation}
where $\Psi \colon \mathbb{R} \times \mathbb{R}^{3} \to
\mathbb{C}^{4}$, $m > 0$ is the mass of the electron, $\gamma^{\mu}$
the $4 \times 4$ Dirac matrices
\begin{equation*}
	\gamma^{0} = \beta = 
	\begin{pmatrix}	
		I & 0 \\
		0 & -I
	\end{pmatrix}
	\qquad
	\gamma^{k} = \beta\alpha^{k}
	\begin{pmatrix}	
		0 & \sigma^{k} \\
		-\sigma^{k} & 0
	\end{pmatrix}, \quad k = 1, \ldots, 3,
\end{equation*}
$\sigma^{k}$ the $2 \times 2$ Pauli matrices
\begin{equation*}
	\sigma^{1} = 	
	\begin{pmatrix}	
		0 & 1 \\
		1 & 0
	\end{pmatrix},
	\qquad
	\sigma^{2} = 	
	\begin{pmatrix}	
		0 & -i \\
		i & 0
	\end{pmatrix},
	\quad
	\sigma^{3} = 	
	\begin{pmatrix}	
		1 & 0 \\
		0 & -1
	\end{pmatrix},
\end{equation*}
and $\gamma \in \mathbb{R}$ is a constant, $\gamma > 0$. The nonlinear
self-interaction of the electron is described by the function $\gamma
\nabla F(\Psi)$. One can find a discussion of the motivation and
history of the models leading to the nonlinear Dirac equation in
\cite{Ranada_1982} where in particular the model nonlinearity
\begin{equation*}
	F(\Psi) = \abs{\scalar{\Psi}{\beta \Psi}}^{\frac{\alpha}{2}} +
	b\abs{\scalar{\Psi}{\gamma^{1} \gamma^{2} \gamma^{3}
	\Psi}}^{\frac{\alpha}{2}}
\end{equation*}
is discussed with $\alpha = 4$ (in case $b = 0$ this is the so called
Soler model).

We prove existence of localized stationary solutions of fixed $L^{2}$
norm for a class of nonlinearities which includes those of the form
above provided $\alpha \in (2,\frac{8}{3}]$ (see below the precise
assumptions).

A normalized stationary solution is a solution of the form
\begin{equation*}
	\Psi(t,x) = \esp^{-i\omega t} \psi(x), \qquad
	\int_{\mathbb{R}^{3}} \abs{\Psi}^{2} = 1.
\end{equation*}
Assuming $F(\esp^{i\theta} \phi) = F(\phi)$ for all $\theta \in 
\mathbb{R}$, $\phi \in \mathbb{C}^{4}$ we have that $\psi$ solves
\begin{equation}
	\label{eq:main}
	\begin{cases}
		\left(-i \malf \cdot \nabla + m \beta - \omega\right) \psi =
		\gamma \nabla F(\psi) & \\
		\int \abs{\psi}^{2} = 1 & 
	\end{cases}.
\end{equation}

We will find solutions for problem \eqref{eq:main} as critical points
of a suitable functional under the constraint of fixed $L^{2}$ norm.
The frequency $\omega$ will arise as the Lagrange multiplier
associated to the constraint. 

The use of variational methods in the study of the existence of
stationary solutions of nonlinear equations and systems involving the
Dirac operator goes back to the pioneering work of Esteban and
S\'er\'e in \cite{Esteban_Sere_1995}, and several later developments,
see for example \cite{Bartsch_Ding_2006, Ding_Wei_2008}. Also related
problems like the Maxwell-Dirac and Klein-Gordon-Dirac problems have
been tackled using variational methods: see the review paper
\cite{Esteban_Lewin_Sere_2008} where many of these results are
presented. In all these works the frequency $\omega$ is assigned, and
the typical result is that for any $\omega \in (0,m)$ there is a
solution of the equation.

The search of normalized solutions, in particular such that $\int
\abs{\psi}^{2} = 1$, is quite natural since $\abs{\Psi(t,x)}^{2}$ can
be interpreted as the probability density of the position of the
electron at point $x$ and time $t$. The problem is also an interesting
mathematical one, and a lot of work has been devoted to finding
normalized solutions for nonlinear Schr\"odinger equations, see
\cite{Stuart_1982, Cazenave_Lions_1982, Berestycki_Lions_1983,
Lions_1984, Lions_1984-1} and more recently \cite{Jeanjean_1997,
Jeanjean_Lu_2019, Hirata_Tanaka_2019, Molle_Riey_Verzini_2022,
Dovetta_Serra_Tilli_2023}. The main approach to finding normalized
solutions for this class of problems is by minimization (usually a
nontrivial one) or other min-max procedures of a suitable functional
on the unit sphere of $L^{2}$.

Very little is known about the existence via variational methods of
normalized solutions for equations involving the Dirac operator. In
this case the functional associated is strongly indefinite and a
direct minimization on the sphere in $L^{2}$ is not possible.

A first result on this topic is the one by Buffoni and Jeanjean
\cite{Buffoni_Jeanjean_1993}, dealing with semilinear elliptic
equations giving rise to a strongly indefinite functional. Later
Esteban and S\'er\'e in \cite{Esteban_Sere_1999} proved the existence
of infinitely many solutions for the Dirac-Fock equation, which
describes the self-interaction of $N$ electrons around a fixed
nucleus. A solution for this problem is a set of $N$ orthonormal
functions which solve an equation with the Dirac operator and a
nonlocal nonlinear term. The orthonormal solutions are found with
variational methods using a penalization method. A similar
penalization method has been used in the paper
\cite{Buffoni_Esteban_Sere_2006} to find normalized solutions for a
strongly indefinite problem not involving the Dirac operator.

Recent results on the existence of normalized solutions for
strongly indefinite problems involving the Dirac operator are
contained in \cite{CotiZelati_Nolasco_2019}, dealing with the problem
of an electron interacting with a nuclei and its own electric field,
\cite{Nolasco_2021} where the existence of a normalized solutions for
the Dirac-Maxwell system is proved, \cite{CotiZelati_Nolasco_2023}
dealing with a Klein-Gordon-Dirac system and \cite{Ding_Yu_Zhao_2023},
where Ding, Yu and Zhao consider a nonlinearity of the form $F(\phi) =
K(x) \abs{\phi}^{\alpha}$, $\alpha \in (2,\frac{8}{3})$ with $K(x) \to
0$ as $\abs{x} \to +\infty$. 

Let us now state our assumptions: we will assume that $F \in
C^{2}(\mathbb{C}^{4},\mathbb{R})$ is such that the following holds for
some $\alpha \in (2,\frac{8}{3}]$ and $R > 0$
\begin{align}
	&\nabla F(0) = D^{2}F(0) = 0, \label{eq:ipotesi0} \tag{H1}\\
	&\abs{D^{2}F(\phi)} \leq \abs{\phi}^{\alpha-2} \quad \text{ for all }
	\abs{\phi} \geq R, \label{eq:ipotesiD2F}\tag{H2}\\
	&\scalar{\nabla F(\phi)}{\phi} \geq \alpha F(\phi) \geq 0 \quad
	\text{ for all } \phi \in \mathbb{C}^{4},
	\label{eq:ipotesiAR}\tag{H3}
\end{align}
and also that, for some $\rho > 0$, $\bar{\gamma} > 0$ and $\nu \in
(\frac{\alpha}{2}, \frac{3}{2})$
\begin{align}
	&F(\phi) \geq \bar{\gamma} \abs{\scalar{\phi}{\beta
	\phi}}^{\frac{\alpha}{2}} \quad \quad \text{ for all }
	\phi \in \mathbb{C}^{4}, \ \abs{\phi} < \rho,
	\label{eq:ipotesiStimaFdalbasso} \tag{H4} \\
	&\abs{\nabla F(\phi)} \leq \abs{\phi}^{\nu} \quad \text{ for all }
	\phi \in \mathbb{C}^{4}, \ \abs{\phi} < \rho. \tag{H5}
	\label{eq:ipotesistimegradbetter}
\end{align}
We also assume that $\xi > 3$ is such that for all $\zeta > 0$ there
is a $C_{\zeta}$ such that for all $\phi \in \mathbb{C}^{4}$
\begin{equation}
	\tag{H6}
	\abs{\nabla F(\phi)} \leq \left( \zeta + C_{\zeta} 
	F(\phi)^{1/\xi}\right) \abs{\phi}. 
	\label{eq:ipotesiES}
\end{equation}

\begin{rem}
	\label{rem:ipotesi}
	Let us remark that the assumptions \eqref{eq:ipotesi0},
	\eqref{eq:ipotesiD2F}, \eqref{eq:ipotesiAR}, \eqref{eq:ipotesiES}
	are the same as in the paper \cite{Esteban_Sere_1995}. The
	restriction on the exponent $\alpha$ is quite natural when looking
	for normalized solutions. The assumption
	\eqref{eq:ipotesiStimaFdalbasso} is similar to assumption (H9) of
	the same paper \cite{Esteban_Sere_1995}. 
	
	Let us also remark that \eqref{eq:ipotesistimegradbetter} holds if
	\begin{equation*}
		\abs{D^{2}F(\phi)} \leq \abs{\phi}^{\alpha-2} \quad \text{ for all }
		\phi \in \mathbb{C}^{4}.
	\end{equation*}
\end{rem}

\begin{rem}
	Follows from \eqref{eq:ipotesi0} and \eqref{eq:ipotesiD2F} that
	for all $\epsilon > 0$ there are $\mu_{\epsilon}$ and
	$\delta_{\epsilon}$ such that, for all $\phi$, $\varphi \in \mathbb{C}^{4}$
	\begin{align}
		&\abs{D^2 F(\phi)} \leq \epsilon + \mu_{\epsilon}
		\abs{\phi}^{\alpha-2} \label{eq:stimeseceps}\\
		&\abs{\nabla F(\phi)} \leq \epsilon \abs{\phi} + \mu_{\epsilon}
		\abs{\phi}^{\alpha-1} \label{eq:stimegradfeps}\\
		&0 \leq F(\phi) \leq \frac{\epsilon}{2} \abs{\phi}^{2} +
		\frac{\mu_{\epsilon}}{\alpha} \abs{\phi}^{\alpha}
		\label{eq:stimeFeps}\\
		&\abs{\nabla F(\phi+\varphi) - \nabla F(\phi)} \leq
		\left(\epsilon + \mu_{\epsilon} (\abs{\phi}^{\alpha-2} +
		\abs{\varphi}^{\alpha-2})\right) \abs{\varphi}
		\label{eq:stimaDiffgradeps}
	\end{align}
	We will write $\mu = \mu_{1}$ and use the fact that, for all
	$\phi$, $\varphi \in \mathbb{C}^{4}$
	\begin{align}
		&\abs{D^2 F(\phi)} \leq 1 + \mu \abs{\phi}^{\alpha-2}
		\label{eq:stimeF_due} \\
		&\abs{\nabla F(\phi)} \leq \abs{\phi} + \mu 
		\abs{\phi}^{\alpha-1} \label{eq:definiscimu}\\
		&0 \leq F(\phi) \leq \frac{1}{2} \abs{\phi}^{2} +
		\frac{\mu}{\alpha} \abs{\phi}^{\alpha} \label{eq:stimeF}\\
		&\abs{\nabla F(\phi+\varphi) - \nabla F(\phi)} \leq \left(1 +
		\mu (\abs{\phi}^{\alpha-2} + \abs{\varphi}^{\alpha-2})\right)
		\abs{\varphi}.
		\label{eq:stimaDiffgrad}
	\end{align}
\end{rem}

Our final assumptions give the size of the coefficient $\gamma$ in
front of the nonlinear term $\nabla F$. We let that $\gamma_{0} > 0$
be such that
\begin{align}
	&\gamma_{0} \left(S_{2}^{2} + \mu S_{3}^{3(\alpha-2)}
	S_{2}^{{3\alpha-8}} \right) < \frac{1}{16},
	\label{eq:ipotesigamma1} \\
	&\gamma_{0} \left(S_{2}^{2} + \mu
	S^{2}_{\frac{4}{4-\alpha}}\right) < \frac{1}{8}
	\label{eq:ipotesigamma2}
\end{align}
where $m > 0$ is the mass of our electron, $\mu$ is the constant in
\eqref{eq:stimeF_due}-\eqref{eq:stimaDiffgrad}, $S_{q}$ is the
constant relative to the Sobolev embedding of $H^{1/2}$ into $L^{q}$
(with the norm \eqref{eq:sobolevnorm}, equivalent to the usual one)
and $\alpha$ is the exponent in \eqref{eq:ipotesiD2F}.

Our result is
\begin{thm}
	\label{thm:main}
	Let $m > 0$, assume $F$ satisfies
	\eqref{eq:ipotesi0}--\eqref{eq:ipotesiES}, and let
	$\gamma \in (0,\gamma_{0}]$.
	
	Then there is $\omega \in (0,m)$ and $\psi \in
	H^{1/2}(\mathbb{R}^{3}, \mathbb{C}^{4})$ solutions of problem
	\eqref{eq:main}.
\end{thm}

As far as we know this is the first result on the existence of
\emph{normalized} solutions for Dirac equation with a nonlinear
interaction of Soler type. Esteban and S\'er\'e in
\cite{Esteban_Sere_1995} deal with essentially the same equation, and
find for all assigned $\omega \in (0,m)$ a solution which is not
normalized. With respect to the result of \cite{Ding_Yu_Zhao_2023} we
have a different nonlinear term, and in particular our result covers
the case in which $F(\phi)$ is a pure power
$\frac{1}{\alpha(\alpha-1)} \abs{\phi}^{\alpha}$ or
$\frac{1}{\alpha(\alpha-1)} \abs{\scalar{\phi}{\beta
\phi}}^{\alpha/2}$, with exponent $\alpha \in (2,\frac{8}{3}]$.

We will find such a solution as a critical point of the functional 
\begin{equation*}
	I(\psi) = \frac{1}{2} \int_{\mathbb{R}^{3}} \scalar{H\psi}{\psi} -
	\gamma \int_{\mathbb{R}^{3}} F(\psi)
\end{equation*}
restricted on the manifold $\norml{\psi}^{2} = 1$.
Here 
\begin{equation*}
	H = -i \malf \cdot \nabla + m\beta.
\end{equation*}

The solutions are found as critical points of the strongly indefinite
functional $I$ restricted to the unit sphere in $L^{2}$, and,
following the method introduced in \cite{CotiZelati_Nolasco_2019,
Nolasco_2021, CotiZelati_Nolasco_2023}, the solution will be found via
the following $\min$-$\max$ procedure: we fix an $L^{2}$ normalized
function $w$ in the positive energy subspace of $H$ and we maximize
$I$ over the ball of radius $1$ in the subspace spanned by $w$ and the
negative energy subspace of $H$. Such a maximum being unique and a
smooth function $\psi(w)$, we then proceed to find a minimizer of $w
\mapsto I(\psi(w))$.

Let us also point out that $F(\phi)$ is not coercive but only
satisfies assumption \eqref{eq:ipotesiStimaFdalbasso}, making it
harder to deduce the estimates necessary to prove the result.

\section{Notation and background results}

We denote with $\abs{u}^{p}_{p} = \int_{\mathbb{R}^{3}}
\abs{u(x)}^{p}$ the norm in $L^{p}(\mathbb{R}^{3}, \mathbb{C}^{4})$
and with $\scalarl{u}{v} = \int_{\mathbb{R}^{3}} \scalar{u(x)}{v(x)}$
the scalar product in $L^{2}(\mathbb{R}^{3}, \mathbb{C}^{4})$, where
$\scalar{\xi}{\eta} = \sum_{i=1}^{4} \bar{\xi}_{i}\eta_{i}$ and
$\abs{\xi}^{2} = \scalar{\xi}{\xi}$ are the scalar product and the
norm in $\mathbb{C}^{4}$. With $\Re z$ we denote the real part of $z 
\in \mathbb{C}$.

We will work in the Hilbert space $X = H^{1/2}(\mathbb{R}^{3}, 
\mathbb{C}^{4})$ with scalar product 
\begin{equation*}
	\scalarh{\psi_{1}}{\psi_{2}} = \int_{\mathbb{R}^{3}}
	\sqrt{\abs{\xi}^{2} + m^{2}} \, \scalar{\hat{\psi}_{1}(\xi)}
	{\hat{\psi}_{2}(\xi)} \, d\xi
\end{equation*}
and corresponding norm 
\begin{equation}
	\label{eq:sobolevnorm}
	\normh{\psi}^{2} = \int_{\mathbb{R}^{3}} \sqrt{\abs{\xi}^{2} + 
	m^{2}}
	\, \abs{\hat{\psi}(\xi)}^{2} \, d\xi.
\end{equation}
where $\hat{\psi}(\xi) = \mathcal{F}\psi(\xi)$ is the Fourier
transform of $\psi \in H^{1/2}(\mathbb{R}^{3}, \mathbb{C}^{4})$. The
norm in \eqref{eq:sobolevnorm} is equivalent to the usual one (given
by \eqref{eq:sobolevnorm} with $m =1$).

Let us also recall some properties (see \cite{Thaller_1992}) of the
Dirac operator
\begin{equation*}
	H = -i \malf \cdot \nabla + m\beta.
\end{equation*}
$H$ is a first order, self-adjoint operator on $H^{1}(\mathbb{R}^{3},
\mathbb{C}^{4})$ with purely absolutely continuous spectrum given by
\begin{equation*}
	\sigma(H) = (-\infty,-m] \cup [m,+\infty).
\end{equation*}
One can define orthogonal projectors $\Lambda_{\pm}$ on the positive
and negative part of the spectrum of $H$. These projections are such
that
\begin{equation*}
	H\Lambda_{\pm} = \Lambda_{\pm} H = \pm \sqrt{-\Delta + m^{2}}
	\Lambda_{\pm} = \pm \Lambda_{\pm} \sqrt{-\Delta + m^{2}}
\end{equation*}
and 
\begin{equation*}
	\scalarh{\Lambda_{+}\psi}{\Lambda_{-}\eta} = 
	\scalarl{\Lambda_{+}\psi}{\Lambda_{-}\eta} = 0
\end{equation*}
so that
\begin{align*}
	\int \scalar{\psi(x)}{H \psi (x)} \, dx &= \int
	\scalar{\Lambda_{+}\psi(x)}{\Lambda_{+}H \psi (x)} \, dx + \int
	\scalar{\Lambda_{-}\psi(x)}{\Lambda_{-}H \psi (x)} \, dx \\
	&= \norml{(-\Delta + 
	m^{2})^{1/4} \Lambda_{+}\psi}^{2} - \norml{(-\Delta + 
	m^{2})^{1/4} \Lambda_{-}\psi}^{2} \\
	&= \normh{\Lambda_{+} \psi}^{2} - \normh{\Lambda_{-} \psi}^{2}.
\end{align*}
We also let $X_{\pm} = \Lambda_{\pm} X$ and, for $\lambda \in (0,1]$,
$\Sigma^{\lambda} = \settc{\psi \in X}{\norml{\psi}^{2} = \lambda}$
and $\Sigma_{\pm}^{\lambda} = \settc{\psi \in
X_{\pm}}{\norml{\psi}^{2} = \lambda}$, $\Sigma_{\pm} =
\Sigma_{\pm}^{1}$.

With $S_{p}$ we will denote the best constant for the Sobolev
embedding of $H^{1/2}(\mathbb{R}^{3},\mathbb{C}^{4})$ (with norm given
by \eqref{eq:sobolevnorm}) in $L^{p}(\mathbb{R}^{3}; \mathbb{C}^{4})$
for $2 \leq p \leq 3$:
\begin{equation*}
	\norml[p]{\psi} \leq S_{p}\normh{\psi}.
\end{equation*}
Let us remark that, with our choice of the norm in $X$, we have that
$m\norml{\psi}^{2} \leq \normh{\psi}^{2}$ and $S_{2} \leq
\frac{1}{\sqrt{m}}$ (actually $S_{2} = \frac{1}{\sqrt{m}}$).

\section{Maximization}

We are interested in solutions of problem \eqref{eq:main}, which are
functions having $L^{2}$ norm equal one. In order to find them, we
will study this problems under the constraint of $L^{2}$ norm equal to
$\lambda$, with $\lambda \in (0,1]$, as in \cite{Lions_1984}.

Let $I \colon H^{1/2}(\mathbb{R}^{3}, \mathbb{C}^{4}) \to \mathbb{R}$ 
\begin{equation*}
	I(\psi) = \frac{1}{2} \normh{\Lambda_{+}\psi}^{2} - \frac{1}{2}
	\normh{\Lambda_{-}\psi}^{2} - \gamma \int_{\mathbb{R}^{3}}
	F(\psi).
\end{equation*}

We denote, for $\lambda \in (0,1]$
\begin{equation*}
	B_{\lambda} = \settc{\eta \in X_{-}}{\norml{\eta}^{2} < \lambda}.
\end{equation*}
For $\eta\in B_{\lambda}$ and $w \in \Sigma_{+}$ we let
\begin{equation*}
	\aeta = \sqrt{\lambda - \norml{\eta}^{2}} \qquad \text{ and }
	\qquad \psi = \aeta w + \eta \in \Sigma^{\lambda}.
\end{equation*}

We will look, given $w \in \Sigma_{+}$, for a maximizer of the
functional $\J$ defined on $B_{\lambda}$
\begin{equation*}
	\J(\eta) = I(\aeta w + \eta) = \frac{1}{2} \normh{\aeta w}^{2} -
	\frac{1}{2} \normh{\eta}^{2} - \gamma \int_{\mathbb{R}^{3}}
	F(\psi).
\end{equation*}

Since $\daeta{\xi} = -\aeta^{-1} \rscalarl{\eta}{\xi}$, the derivative
of $\J$ is given, for all $\xi \in X_{-}$, by
\begin{multline}
	\label{eq:derivataJ}
	\dJ{\xi} = dI(\aeta w + \eta)[\daeta{\xi}w + \xi] =
	dI(\psi)[h_{\xi}] \\
	= \rscalarh{\aeta w}{\daeta{\xi} w} - \rscalarh{\eta}{\xi} - \gamma
	\int_{\mathbb{R}^{3}} \scalar{\nabla F(\psi)}{h_{\xi}}\\
	= -\rscalarl{\eta}{\xi} \normh{w}^{2} - \rscalarh{\eta}{\xi} -
	\gamma \int_{\mathbb{R}^{3}} \scalar{\nabla F(\psi)}{h_{\xi}}
\end{multline}
(here $h_{\xi} = \daeta{\xi}w + \xi$) and, in particular
\begin{equation*}
	\dJ{\eta} = -\norml{\eta}^{2} \normh{w}^{2} - \normh{\eta}^{2} -
	\gamma \int_{\mathbb{R}^{3}} \scalar{\nabla F(\psi)}{h_{\eta}}.
\end{equation*}

\begin{lem}
	\label{lem:stime_grad}
	For all $w \in \Sigma_{+}$ and $\eta \in B_{\lambda}$ we have
	\begin{equation}
		\label{eq:bdd}
		\normh{\eta}^{2} \leq \aeta^{2} \normh{w}^{2} - 2\J(\eta),
	\end{equation}
	and for all $\eta \in B_{\lambda}$ such that $\J(\eta) \geq 0$ 
	we have that
	\begin{equation}
		\label{eq:stimaFw}
		\frac{1}{a(\eta)} \labs{\int \scalar{\nabla F(\psi)}{w}} \leq
		C_{\lambda,\alpha} \normh{w}^{2},
	\end{equation}
	where 
	\begin{equation*}
		C_{\lambda,\alpha} = 4 \left(S_{2}^{2} + \mu
		\lambda^{\frac{\alpha-2}{2}} S_{3}^{3(\alpha-2)}
		S_{2}^{{3\alpha-8}} \right).
	\end{equation*}
	
	Moreover, if $\J(\eta) \geq 0$, $\norml{\eta}^{2} \geq
	\frac{\lambda}{2}$ we have that
	\begin{equation}
		\label{eq:spinge}
		\dJ{\eta} < - \normh{\eta}^{2} < 0.
	\end{equation}
\end{lem}

\begin{proof}	
	We have that for $\eta \in B_{\lambda}$ and $\psi = \aeta w +
	\eta$
	\begin{equation*}
		\frac{1}{2} \normh{\eta}^{2} \leq \frac{1}{2}
		\normh{\eta}^{2} +  \gamma \int F(\psi)
		= \frac{1}{2} \normh{a(\eta)w}^{2} - \J(\eta)
	\end{equation*}
	and \eqref{eq:bdd} follows.

	Let us now assume that $\J(\eta) \geq 0$, so that
	$\normh{\eta}^{2} \leq \aeta^{2} \normh{w}^{2}$.
	
	Since 
	\begin{equation*}
		a(\eta) - da(\eta)[\eta] = a(\eta) +
		a(\eta)^{-1}\norml{\eta}^{2} = \lambda a(\eta)^{-1}
	\end{equation*}
	we have that
	\begin{equation*}
		\dJ{\eta} = -\norml{\eta}^{2} \normh{w}^{2} - \normh{\eta}^{2}
		- \gamma \int \scalar{\nabla F(\psi)}{\psi} 
		+ \gamma \frac{\lambda}{a(\eta)^{2}} \int \scalar{\nabla
		F(\psi)}{a(\eta)w}.
	\end{equation*}
	
	Let's estimate the last term. Using \eqref{eq:definiscimu} we have
	\begin{multline*}
		\labs{\int \scalar{\nabla F(\psi)}{a(\eta)w}} \leq \int
		\left(\abs{\psi} + \mu \abs{\psi}^{\alpha-1}\right)
		\abs{a(\eta)w} \\
		\leq \norml{\psi} \norml{a(\eta)w} + \mu
		\norml[\alpha]{\psi}^{\alpha-1} \norml[\alpha]{a(\eta)w}.
	\end{multline*}
	Since $\norml{\psi}^{2} = \lambda$ and $\norml{a(\eta)w}^{2} \leq
	\lambda$ we deduce that
	\begin{equation*}
		\norml[\alpha]{\psi}^{\alpha} = \int \abs{\psi}^{\alpha} \leq
		\int \abs{\psi}^{3\alpha-6} \abs{\psi}^{6-2\alpha} \leq
		\norml[3]{\psi}^{3\alpha-6} \norml[2]{\psi}^{6-2\alpha} =
		\lambda^{3-\alpha} \norml[3]{\psi}^{3\alpha-6}
	\end{equation*}
	and 
	\begin{equation*}
		\norml[\alpha]{a(\eta)w}^{\alpha} \leq \lambda^{3-\alpha}
		\norml[3]{a(\eta)w}^{3\alpha-6}
	\end{equation*}
	so that
	\begin{align*}
		&\labs{\int \scalar{\nabla F(\psi)}{a(\eta)w}} \leq
		\norml{\psi} \norml{a(\eta)w} + \mu \lambda^{(3-\alpha)}
		\norml[3]{\psi}^{\frac{3(\alpha-2)(\alpha-1)}{\alpha}}
		\norml[3]{a(\eta)w}^{\frac{3(\alpha-2)}{\alpha}}\\
		&\qquad\leq S_{2}^{2} \normh{\psi} \normh{a(\eta)w} + \mu
		S_{3}^{3(\alpha-2)} \lambda^{(3-\alpha)}
		\normh{\psi}^{\frac{3(\alpha-2)(\alpha-1)}{\alpha}}
		\normh{a(\eta)w}^{\frac{3(\alpha-2)}{\alpha}} \\
		&\qquad\leq S_{2}^{2} \normh{\psi}^{2} + \mu S_{3}^{3(\alpha-2)}
		\lambda^{(3-\alpha)} \normh{\psi}^{3(\alpha-2)}.
	\end{align*}
	Since $\normh{\psi}^{2} \geq S_{2}^{2}\norml{\psi}^{2} = S_{2}^{2}
	\lambda$ and $3(\alpha - 2) \in (0,2]$ for all $\alpha \in (2,
	\frac{8}{3}]$ we have that
	\begin{equation*}
		\normh{\psi}^{3(\alpha-2)} \leq
		(S_{2}\sqrt{\lambda})^{{3\alpha-8}} \normh{\psi}^{2} \leq 4
		(S_{2}\sqrt{\lambda})^{{3\alpha-8}} a(\eta)^{2} \normh{w}^{2}
	\end{equation*}
	and
	\begin{equation*}
		\frac{\lambda}{a(\eta)^{2}} \labs{\int \scalar{\nabla
		F(\psi)}{a(\eta) w} }\leq 4 \lambda \left(S_{2}^{2} + \mu
		\lambda^{\frac{\alpha-2}{2}} S_{3}^{3(\alpha-2)}
		S_{2}^{{3\alpha-8}} \right) \normh{w}^{2} = \lambda
		C_{\alpha,\lambda} \normh{w}^{2}.
	\end{equation*}
	We finally deduce that
	\begin{equation*}
		\dJ{\eta} \leq \left(-\norml{\eta}^{2} + \gamma \lambda
		C_{\alpha,\lambda} \right) \normh{w}^{2} - \normh{\eta}^{2} -
		\gamma \int \scalar{\nabla F(\psi)}{\psi} < -\normh{\eta}^{2}
	\end{equation*}
	if $\norml{\eta}^{2} > \frac{\lambda}{2}$ and since by 
	\eqref{eq:ipotesigamma1} $\gamma$ is
	such that $\gamma C_{\alpha,\lambda} < \frac{1}{4}$ for all 
	$\lambda \in (0,1]$.
\end{proof}

\begin{rem}
	\label{rem:psbdd}
	It follows from the Lemma \ref{lem:stime_grad} that if $\eta_{n}$
	is a Palais-Smale sequence for $\J$ such that $\J(\eta_{n}) \geq
	0$, then $\norml{\eta_{n}}^{2} < \frac{\lambda}{2}$ for all $n \in
	\mathbb{N}$ large enough.
\end{rem}

\begin{lem}
	\label{lem:PSvale}
	
	Let $\eta_{n} \in B_{\lambda}$ be a Palais-Smale sequence for
	$\J$, that is
	\begin{equation*}
		\J(\eta_{n}) \to c \geq 0, \qquad d\J(\eta_{n}) \to 0.
	\end{equation*}
	
	Then $\eta_{n}$ converges, up to a subsequence, to a critical 
	point $\eta$ of $\J$.
\end{lem}

\begin{proof}
	Follows form the Lemma \ref{lem:stime_grad} and Remark
	\ref{rem:psbdd} that $\norml{\eta_{n}}^{2} < \frac{\lambda}{2}$
	and that $\normh{\eta_{n}} \leq a(\eta_{n})\normh{w} \leq
	\sqrt{\lambda} \normh{w}$ is bounded, hence $\eta_{n} \tow \eta$
	(up to a subsequence).
	
	From
	\begin{align*}
		o(1) &= d\J(\eta_{n})[\eta_{n} - \eta] = -
		\rscalarl{\eta_{n}}{\eta_{n} - \eta} \normh{w}^{2} -
		\rscalarh{\eta_{n}}{\eta_{n} - \eta} \\
		&\qquad - \gamma \int \scalar{\nabla
		F(\psi_{n})}{da(\eta_{n})[\eta_{n} - \eta]w + \eta_{n} - \eta}
		\\
		&= - \norml{\eta_{n} - \eta}^{2} \normh{w}^{2} -
		\rscalarl{\eta}{\eta_{n} - \eta} \normh{w}^{2} -
		\normh{\eta_{n} - \eta}^{2} - \rscalarh{\eta}{\eta_{n} - \eta}
		\\
		&\qquad - \gamma \int \scalar{\nabla F(\psi_{n})}{ \eta_{n} -
		\eta} - \gamma da(\eta_{n})[\eta_{n} - \eta] \int
		\scalar{\nabla F(\psi_{n})}{w}
	\end{align*}
	we deduce, using also the fact that $\norml{\eta_{n}}^{2} <
	\frac{\lambda}{2}$, that
	\begin{multline*}
		\norml{\eta_{n} - \eta}^{2} \normh{w}^{2} + \normh{\eta_{n} -
		\eta}^{2} \\
		= - \gamma \int \scalar{\nabla F(\psi_{n})}{ \eta_{n} - \eta}
		+ \gamma a(\eta_{n})^{-1}\rscalarl{\eta_{n}}{\eta_{n} - \eta}
		\int \scalar{\nabla F(\psi_{n})}{w} + o(1)\\
		= - \gamma \int \scalar{\nabla F(\psi_{n})}{ \eta_{n} - \eta}
		+ \gamma a(\eta_{n})^{-1}\norml{\eta_{n} - \eta}^{2} \int
		\scalar{\nabla F (\psi_{n})}{w} + o(1).
	\end{multline*}
	From the estimate \eqref{eq:stimaFw} follows
	\begin{equation*}
		(1 - \gamma C_{\alpha, \lambda}) \norml{\eta_{n} - \eta}^{2}
		\normh{w}^{2} + \normh{\eta_{n} - \eta}^{2} \leq - \gamma \int
		\scalar{\nabla F(\psi_{n})}{ \eta_{n} - \eta} + o(1).
	\end{equation*}

	We will prove that $\liminf_{n\to +\infty} \int \scalar{\nabla F
	(\psi_{n})}{ \eta_{n} - \eta} \geq 0$ applying concentration
	compactness to the sequence of functions $\eta_{n} - \eta$.
	
	In case we have vanishing, that is
	\begin{equation*}
		\forall R > 0 \qquad \lim_{n \to +\infty} \sup_{y \in
		\mathbb{R}^{3}} \int_{B_{R}(y)} \abs{\eta_{n} - \eta}^{2} = 0,
	\end{equation*}
	we have that $\norml[q]{\eta_{n} - \eta} \to 0$ as $n \to +\infty$
	for all $q \in (2,3)$ (see \cite{Lions_1984} or \cite[Lemma
	1.21]{Willem_1996}). 
	
	For all $\epsilon > 0$ we have that
	\begin{align*}
		\labs{\int \scalar{\nabla F(\psi_{n})}{\eta_{n} - \eta}} &\leq
		\int \left(\epsilon\abs{\psi_{n}} + \mu_{\epsilon}
		\abs{\psi_{n}}^{\alpha-1} \right) \abs{\eta_{n} - \eta} \\
		&\leq \epsilon \norml{\eta_{n} - \eta} + \mu_{\epsilon}
		\norml[\alpha]{\psi_{n}}^{\alpha-1} \norml[\alpha]{\eta_{n} -
		\eta} \leq \sqrt{2 \lambda} \epsilon + o(1)
	\end{align*}
	so that $\lim_{n\to +\infty} \int \scalar{\nabla F (\psi_{n})}{
	\eta_{n} - \eta} = 0$ if vanishing occurs.

	If we have dichotomy then for all $\epsilon > 0$ there is a
	sequences $R_{n} \to +\infty$ such that
	\begin{equation*}
		\int_{\mathbb{R}^{3} \setminus B_{R_{n}}(0)} \norml{\eta_{n} -
		\eta}^{2} \geq \sigma - \epsilon,
	\end{equation*} 
	where (up to a subsequence)
	\begin{equation*}
		\sigma = \lim_{n \to + \infty} \int_{\mathbb{R}^{3}} \abs{\eta_{n}
		- \eta}^{2}.
	\end{equation*}
	
	We have that
	\begin{align*}
		\int &\scalar{\nabla F(\psi_{n})}{ \eta_{n} - \eta} \\
		&= \int_{B_{R_{n}}(0)} \scalar{\nabla F(\psi_{n})}{
		\eta_{n} - \eta} + \int_{\mathbb{R}^{3} \setminus
		B_{R_{n}}(0)} \scalar{\nabla F(\eta_{n})}{\eta_{n}} \\
		&\qquad + \int_{\mathbb{R}^{3} \setminus B_{R_{n}}(0)}
		\scalar{\nabla F(\psi_{n}) - \nabla F(\eta_{n})}{\eta_{n}} -
		\int_{\mathbb{R}^{3} \setminus B_{R_{n}}(0)} \scalar{\nabla
		F(\psi_{n})}{\eta}\\
		&\geq -c(\sqrt{\epsilon} +
		\epsilon^{\frac{2(3-\alpha)}{\alpha}}) \\
		&\qquad + \int_{\mathbb{R}^{3} \setminus B_{R_{n}}(0)}
		\scalar{\nabla F(\psi_{n}) - \nabla F(\eta_{n})}{\eta_{n}} -
		\int_{\mathbb{R}^{3} \setminus B_{R_{n}}(0)} \scalar{\nabla
		F(\psi_{n})}{\eta}
	\end{align*}
	where we have used the fact that $\scalar{\nabla
	F(\eta_{n})}{\eta_{n}} \geq 0$ and
	\begin{align*}
		&\labs{\int_{B_{R_{n}}(0)} \scalar{\nabla F(\psi_{n})}{
		\eta_{n} - \eta}} \\
		&\quad \leq \norml{\psi_{n}} \left(\int_{B_{R_{n}}(0)}
		\abs{\eta_{n} - \eta}^{2}\right)^{\frac{1}{2}} + \mu
		\norml[\alpha]{\psi_{n}}^{\alpha-1} \left(\int_{B_{R_{n}}(0)}
		\abs{\eta_{n} - \eta}^{\alpha}\right)^{\frac{1}{\alpha}} \\
		&\quad \leq \left(\int_{B_{R_{n}}(0)} \abs{\eta_{n} -
		\eta}^{2}\right)^{\frac{1}{2}} + \mu
		\norml[\alpha]{\psi_{n}}^{\alpha-1}
		\norml[3]{\eta_{n}-\eta}^{\frac{3(\alpha-2)}{\alpha}} \left(
		\int_{B_{R_{n}}(0)} \abs{\eta_{n} - \eta}^{2}
		\right)^{\frac{2(3-\alpha)}{\alpha}} \\
		&\quad \leq c(\sqrt{\epsilon} +
		\epsilon^{\frac{2(3-\alpha)}{\alpha}})
	\end{align*}
	for some constant $c > 0$.
	
	We now observe that
	\begin{align*}
		\int_{\mathbb{R}^{3} \setminus B_{R_{n}}(0)} &\scalar {\nabla
		F(\psi_{n}) - \nabla F(\eta_{n})}{\eta_{n}} \leq
		\int_{\mathbb{R}^{3} \setminus B_{R_{n}}(0)} \abs{\nabla
		F(\psi_{n}) - \nabla F(\eta_{n})}\abs {\eta_{n}} \\
		&\leq \int_{\mathbb{R}^{3} \setminus B_{R_{n}}(0)}
		a(\eta_{n}) \abs{w} \abs {\eta_{n}} \\
		&\qquad + \mu \int_{\mathbb{R}^{3} \setminus B_{R_{n}}(0)}
		\left(\abs{\eta_{n}}^{\alpha-2} + a(\eta_{n})^{\alpha-2}
		\abs{w}^{\alpha-2} \right) a(\eta_{n}) \abs{w} \abs {\eta_{n}}
		\\
		& \leq \sqrt{\lambda} \norml{\eta_{n}}
		\left(\int_{\mathbb{R}^{3} \setminus B_{R_{n}(0)}} \abs{w}^{2}
		\right)^{\frac{1}{2}} + \mu\sqrt{\lambda}
		\norml[\alpha]{\eta_{n}}^{\alpha-1} \left(\int_{\mathbb{R}^{3}
		\setminus B_{R_{n}}(0)} \abs{w}^{\alpha}\right)^{1/\alpha} \\
		& \qquad + \mu \lambda^{\frac{\alpha-1}{2}}
		\norml[\alpha]{\eta_{n}} \left(\int_{\mathbb{R}^{3} \setminus
		B_{R_{n}}(0)} \abs{w}^{\alpha}\right)^{(\alpha-1)/\alpha} =
		o(1)
	\end{align*}
	where we have used \eqref{eq:stimaDiffgrad}.
	
	A similar argument shows that 
	\begin{equation*}
		\int_{\mathbb{R}^{3} \setminus B_{R_{n}}(0)}
		\scalar{\nabla F(\psi_{n})}{\eta} = o(1),
	\end{equation*}
	and also when dichotomy occurs we find that 
	\begin{equation*}
		(1 - \gamma C_{\alpha,\lambda}) \norml{\eta_{n} - \eta}^{2}
		\normh{w}^{2} + \normh{\eta_{n} - \eta}^{2} \leq c\epsilon +
		o(1).
	\end{equation*}
	
	So in every case we have that and $\eta_{n} \to \eta$, with $\eta$
	critical point of $\J$.
\end{proof}

We now show that the functional $\J$ is concave in the set where it is
positive in the $L^{2}$-ball of radius $\sqrt{\frac{\lambda}{2}}$.
\begin{lem}
	\label{lem:criticomax}
	Let $\eta \in B_{\lambda}$ be such that $\norml{\eta}^{2} < 
	\frac{\lambda}{2}$ and $\J(\eta) \geq 0$. Then
	\begin{equation*}
		d^{2}\J(\eta)[\xi, \xi] \leq - \frac{3}{4} \normh{\xi}^{2}
		\qquad \text{ for all } \xi \in X_{-}.
	\end{equation*}
\end{lem}

\begin{proof}
	In order to compute the second derivative we denote $\psi = \aeta
	w + \eta$ and $h_{\xi} = \daeta{\xi}w + \xi$ and observe that
	\begin{equation*}
		d^{2}\aeta[\xi,\zeta] = -\aeta^{-1} \left( \Re
		\scalarl{\xi}{\zeta} + \frac{\Re \scalarl{\eta}{\xi} \Re
		\scalar{\eta}{\zeta}}{\lambda - \norml{\eta}^{2}} \right).
	\end{equation*}
	We have
	\begin{align*}
		d^{2}\J(\eta)[\xi, \xi] &= - \norml{\xi}^{2}\normh{w}^{2} -
		\normh{\xi}^{2} - \gamma \int \scalar{D^2
		F(\psi)h_{\xi}}{h_{\xi}} \\
		&\qquad - \gamma \int \scalar{\nabla
		F(\psi)}{d^{2}a(\eta)[\xi,\xi]w}\\
		&= - \norml{\xi}^{2}\normh{w}^{2} - \normh{\xi}^{2} - \gamma
		\int \scalar{D^2 F(\psi)h_{\xi}}{h_{\xi}} \\
		&\qquad + \frac{\gamma}{\aeta} \left( \norml{\xi}^{2} +
		\frac{\scalarl{\eta}{\xi}^{2}}{\lambda - \norml{\eta}^{2}}
		\right) \int \scalar{\nabla F(\psi)}{w} \\
	\end{align*}
	Then, using \eqref{eq:stimaFw} and \eqref{eq:stimeF_due}:
	\begin{align*}
		d^{2}\J(\eta)[\xi, \xi] &\leq - \norml{\xi}^{2}\normh{w}^{2} -
		\normh{\xi}^{2} + \gamma \int |h_{\xi}|^2 + \gamma \mu \int
		|\psi|^{\alpha-2} |h_{\xi}|^2 \\
		&\qquad + \frac{\lambda \norml{\xi}^{2} - \norml{\xi}^{2}
		\norml{\eta}^{2} + \scalarl{\eta}{\xi}^{2}}{\lambda -
		\norml{\eta}^{2}} \gamma C_{\alpha,\lambda} \normh{w}^{2} \\
		&\leq - (1- \gamma S_2^2)\norml{\xi}^{2}\normh{w}^{2} -(1-
		\gamma S_2^2 ) \normh{\xi}^{2} \\
		&\qquad + \gamma \mu |\psi|_2^{\alpha-2} | h_{\xi}
		|^2_{\frac{4}{4-\alpha}}+ 2 \norml{\xi}^{2} \gamma
		C_{\alpha,\lambda} \normh{w}^2 \\
		&\leq - \left(1 - \gamma (2C_{\alpha,\lambda} +S_2^2 + \mu
		\lambda^{\frac{\alpha-2}{2}} S_{\frac{4}{4-\alpha}}^2)
		\right)\norml{\xi}^{2} \normh{w}^{2} \\
		& \qquad - \left(1 - \gamma ( S_2^2 +
		\mu\lambda^{\frac{\alpha-2}{2}} S_{\frac{4}{4-\alpha}}^2)
		\right)\normh{\xi}^{2} \\
		&\leq - \frac{1}{4}\norml{\xi}^{2} \normh{w}^{2} - \frac{3}{4}
		\normh{\xi}^{2} \leq - \frac{3}{4}\normh{\xi}^{2}.
	\end{align*}
	since $2 \gamma C_{\alpha,\lambda} < \frac{1}{2}$ by 
	\eqref{eq:ipotesigamma1} and by \eqref{eq:ipotesigamma2}.
\end{proof}

\begin{lem}
	\label{lem:stimautile}
	Let $w \in \Sigma_{+}$, $\eta \in B_{\lambda}$ and $\psi
	= a(\eta)w + \eta \in \Sigma^{\lambda}$. Then we have that
	\begin{multline}
		\int F(\psi) \geq \int F(\sqrt{\lambda}w) + \int
		\scalar{\nabla F(aw)}{\eta} - \left(S_{2}^{2} + \mu
		\lambda^{\frac{\alpha-2}{2}} S^{2}_{\frac{4}{4-\alpha}}\right)
		\norml{\eta}^{2}\normh{w}^{2} \\
		- \left(S_{2}^{2} + \mu \lambda^{\frac{\alpha-2}{2}}
		S^{\alpha}_{\frac{4}{4-\alpha}} \right) \normh{\eta}^{2} - \mu
		\lambda^{\frac{\alpha - 2}{2}} \int \abs{w}^{\alpha-2}
		\abs{\eta}^{2}
	\end{multline}
\end{lem}

\begin{proof}
	The result follows, using \eqref{eq:stimaDiffgrad} from 
	\begin{align*}
		\int F(\psi) &- \int F(a(\eta)w)) - \int \scalar{\nabla
		F(a(\eta)w)}{\eta} = \\
		&=\int \scalar{\nabla F(a(\eta)w + \theta\eta) - \nabla
		F(a(\eta)w)}{\eta} \\
		&\geq - \int (1 + \mu \abs{a(\eta)w}^{\alpha-2} + \mu 
		\abs{\eta}^{\alpha-2})\abs{\eta}^{2} \\
		&\geq - \int \abs{\eta}^{2} - \mu \lambda^{\frac{\alpha-2}{2}}
		\int \abs{w}^{\alpha-2} \abs{\eta}^{2} - \mu \int
		\abs{\eta}^{\alpha} \\
		&\geq - \norml{\eta}^{2} - \mu \lambda^{\frac{\alpha-2}{2}}
		\int \abs{w}^{\alpha-2} \abs{\eta}^{2} - \mu
		\norml{\eta}^{\alpha-2} \norml[\frac{4}{4-\alpha}]{\eta}^{2} 
		\\
		&\geq - (S_{2}^{2} + \mu \lambda^{\frac{\alpha - 2}{2}}
		S^{2}_{\frac{4}{4-\alpha}}) \normh{\eta}^{2} - \mu
		\lambda^{\frac{\alpha-2}{2}} \int \abs{w}^{\alpha-2}
		\abs{\eta}^{2}
	\end{align*}
	for some $\theta(x) \in [0,1]$ and 
	\begin{align*}
		\int &(F(a(\eta)w) - F(\sqrt{\lambda} w)) = \int
		\scalar{\nabla F(a(\eta) w + \theta(a(\eta) - \sqrt{\lambda})
		w)}{(a(\eta) - \sqrt{\lambda}) w} \\
		&\geq - \int \abs{a(\eta) + \theta(a(\eta) - \sqrt{\lambda})}
		\abs{w} \abs{a(\eta) - \sqrt{\lambda}} \abs{w} \\
		&\qquad - \mu\int \abs{a(\eta) + \theta(a(\eta) -
		\sqrt{\lambda})}^{\alpha-1} \abs{w}^{\alpha-1} \abs{a(\eta) -
		\sqrt{\lambda}} \abs{w} \\
		&\geq - \sqrt{\lambda} \abs{\sqrt{\lambda} - a(\eta)} \int
		\abs{w}^{2} - \mu \lambda^{\frac{\alpha-1}{2}}
		\abs{\sqrt{\lambda} - a(\eta)} \int \abs{w}^{\alpha}\\
		&\geq - (\sqrt{\lambda} - a(\eta))
		\left(\sqrt{\lambda}\norml{w}^{2} + \mu
		\lambda^{\frac{\alpha-1}{2}} \int \abs{w}^{\alpha} \right) \\
		&\geq - \frac{\lambda - a(\eta)^{2}}{\sqrt{\lambda} + a(\eta)}
		\left(\sqrt{\lambda}\norml{w}^{2} + \mu
		\lambda^{\frac{\alpha-1}{2}} \norml{w}^{(\alpha-2)}
		\norml[4/(4-\alpha)]{w}^{2} \right) \\
		&\geq - \left(S_{2}^{2} + \mu \lambda^{\frac{\alpha-2}{2}}
		S^{2}_{\frac{4}{4-\alpha}}\right)
		\norml{\eta}^{2}\normh{w}^{2}
	\end{align*}
\end{proof}

We let, for all $w \in \Sigma_{+}$
\begin{equation*}
	\E(w) = \sup_{\eta \in B_{\lambda}} \J(\eta).
\end{equation*}

\begin{lem}
	\label{lem:stime_max}
	For all $w \in \Sigma_{+}$ we have
	\begin{multline*}
		0 < \frac{\lambda}{2 S_{2}^{2}} \left(1 - \gamma
		\left(S_{2}^{2} + 2\frac{\mu}{\alpha}
		\lambda^{\frac{\alpha-2}{2}}
		S^{2}_{\frac{4}{4-\alpha}}\right)\right) \\
		\leq \frac{\lambda}{2} \left(1 - \gamma \left(S_{2}^{2} +
		2\frac{\mu}{\alpha} \lambda^{\frac{\alpha-2}{2}}
		S^{2}_{\frac{4}{4-\alpha}}\right)\right) \normh{w}^{2} \leq
		\E(w) \leq \frac{\lambda}{2} \normh{w}^{2}.
	\end{multline*}
\end{lem}

\begin{proof}
	We have, using \eqref{eq:stimeF}, that 
	\begin{align*}
		\E(w) &\geq \J(0) = \frac{1}{2} \normh{\sqrt{\lambda}w}^{2} -
		\gamma \int F(\sqrt{\lambda} w) \\
		&\geq \frac{\lambda}{2} \normh{w}^{2} - \gamma
		\frac{\lambda}{2} \norml{w}^{2} - \gamma \frac{\mu}{\alpha}
		\lambda^{\frac{\alpha}{2}} \norml{w}^{\alpha-2}
		\norml[\frac{4}{4-\alpha}]{w}^{2} \\
		& \geq \frac{\lambda}{2} \normh{w}^{2} - \gamma \lambda
		\left(\frac{S_{2}^{2}}{2} + \frac{\mu}{\alpha}
		\lambda^{\frac{\alpha-2}{2}} S^{2}_{\frac{4}{4-\alpha}}\right)
		\normh{w}^{2} \\
		& \geq \frac{\lambda}{2 S_{2}^{2}} \left(1 - \gamma
		\left(S_{2}^{2} + 2\frac{\mu}{\alpha}
		\lambda^{\frac{\alpha-2}{2}}
		S^{2}_{\frac{4}{4-\alpha}}\right)\right)
	\end{align*}
	which is positive by \eqref{eq:ipotesigamma2} and
	\begin{equation*}
		\J(\eta) = \frac{1}{2} \normh{\aeta w}^{2} - \frac{1}{2}
		\normh{\eta}^{2} - \gamma \int F(\psi) \leq \frac{\lambda}{2}
		\normh{w}^{2}.
	\end{equation*}
\end{proof}

The following proposition is analogous to \cite[Proposition
4.5]{CotiZelati_Nolasco_2019}, \cite[Proposition 3.6]{Nolasco_2021}
and \cite[Proposition 2.9]{CotiZelati_Nolasco_2023}. Since in the
present situation the functional $\J$ is, by lemma
\ref{lem:criticomax}, concave in the region where $\J(\eta) \geq 0$
and $\norml{\eta}^{2} < \frac{\lambda}{2}$ the proof is actually
simpler.
\begin{prop}
	\label{prop:massimo}
	For every $w \in \Sigma_{+}$ and $\lambda > 0$ there is a unique
	$\eta(w) \in B_{\lambda}$ such that
	\begin{equation*}
		\J(\eta(w)) = \max_{\eta \in B_{\lambda}} \J(\eta) = \E(w).
	\end{equation*}
	$\eta(w)$ is a critical point of $\J$ on $B_{\lambda}$ such that
	$\norml{\eta(w)}^{2} < \frac{\lambda}{2}$ and
	\begin{equation}
		\label{eq:stime_maxH}
		\normh{\eta(w)}^{2} \leq \lambda \gamma \left(S_{2}^{2} +
		2\frac{\mu}{\alpha} \lambda^{\frac{\alpha-2}{2}}
		S^{2}_{\frac{4}{4-\alpha}}\right) \normh{w}^{2}.
	\end{equation}
	
	Moreover the map 
	\begin{equation*}
		w \in X_{+} \setminus \{0\} \mapsto \gamma(w) =
		\eta(\norml{w}^{-1}w) \in B_{\lambda}
	\end{equation*} 
	is smooth.
\end{prop}
	
\begin{proof}
	We can find, by lemma \ref{lem:stime_max} and using Ekeland's
	variational principle a maximizing Palais-Smale sequence
	$\eta_{n}$ at positive level.
	
	Then, by lemma \ref{lem:PSvale}, $\eta_{n} \to \eta$ (up to a 
	subsequence), with 
	\begin{equation*}
		d\J(\eta) = 0, \qquad \J(\eta) = \E(w) > 0.
	\end{equation*}
	Using Lemma \ref{lem:stime_max} we deduce that
	\begin{equation*}
		\frac{\lambda}{2} \left(1 - \gamma \left(S_{2}^{2} +
		2\frac{\mu}{\alpha} \lambda^{\frac{\alpha-2}{2}}
		S^{2}_{\frac{4}{4-\alpha}}\right)\right) \normh{w}^{2} \leq
		\E(w) = \J(\eta) \leq \frac{\lambda}{2} \normh{w}^{2} -
		\frac{1}{2} \normh{\eta}^{2}
	\end{equation*}
	from which \eqref{eq:stime_maxH} follows.
	
	As a consequence we have that for all $w \in \Sigma_{+}$ the
	convex set
	\begin{equation*}
		\mathcal{B} = \settc{\eta \in B_{\lambda}}{\norml{\eta}^{2} <
		\frac{1}{2} \text{ and } \normh{\eta}^{2} \leq 2\lambda \gamma
		\left(S_{2}^{2} + 2\frac{\mu}{\alpha}
		\lambda^{\frac{\alpha-2}{2}} S^{2}_{\frac{4}{4-\alpha}}\right)
		\normh{w}^{2}}
	\end{equation*}
	contains all the maxima of $\J$ in $B_{\lambda}$. We now show 
	that $\J(\eta) \geq 0$ for all $\eta \in \mathcal{B}$. 
	
	We have that
	\begin{align*}
		\int F(a(\eta)w + \eta) &\leq \frac{1}{2}\norml{a(\eta)w + \eta}^{2} +
		\frac{\mu}{\alpha} \norml[\alpha]{a(\eta)w + \eta}^{\alpha} \\
		&\leq \frac{\lambda}{2} + \frac{\mu}{\alpha}
		{\lambda}^{\frac{\alpha-2}{2}} \norml[\frac{4}{4-\alpha}]{a(\eta)w
		+ \eta}^{2} \\
		&\leq \frac{\lambda}{2}S_{2}^{2}\normh{w}^{2} +
		\frac{\mu}{\alpha} {\lambda}^{\frac{\alpha-2}{2}}
		S_{\frac{4}{4-\alpha}}^{2} \normh{a(\eta)w + \eta}^{2} \\
		&\leq \left(\frac{\lambda}{2}S_{2}^{2} + \frac{\mu}{\alpha}
		{\lambda}^{\frac{\alpha}{2}} S_{\frac{4}{4-\alpha}}^{2} +
		\frac{1}{2} \frac{\mu}{\alpha} {\lambda}^{\frac{\alpha}{2}}
		S_{\frac{4}{4-\alpha}}^{2}\right)\normh{w}^{2} \\
		&\leq \lambda \left(\frac{1}{2}S_{2}^{2} + \mu
		S_{\frac{4}{4-\alpha}}^{2} \right)\normh{w}^{2}
	\end{align*}
	
	We then have
	\begin{align*}
		\J(\eta) &= \frac{1}{2} a(\eta)^{2} \normh{w}^{2} - 
		\frac{1}{2} \normh{\eta}^{2} - \gamma \int F(a(\eta)w + \eta) \\
		&\geq \frac{\lambda}{4} \normh{w}^{2} - \lambda \gamma
		\left(S_{2}^{2} + 2\frac{\mu}{\alpha}
		\lambda^{\frac{\alpha-2}{2}} S^{2}_{\frac{4}{4-\alpha}}\right)
		\normh{w}^{2} - \gamma \lambda \left(\frac{1}{2}S_{2}^{2} +
		\mu S_{\frac{4}{4-\alpha}}^{2} \right)\normh{w}^{2} \\
		&> \frac{\lambda}{4} \normh{w}^{2} - \frac{\lambda}{8}
		\normh{w}^{2} - \frac{\lambda}{8} \normh{w}^{2} = 0
	\end{align*}

	The strict concavity of $\J$ in the convex region $\mathcal{B}$
	implies that such a maximum is unique.
	
	The properties of the map $w \mapsto \gamma(w) =
	\eta(\norml{w}^{-1}w)$ follows exactly as in \cite[Proposition
	2.9]{CotiZelati_Nolasco_2023}.
\end{proof}

Exactly as in \cite{CotiZelati_Nolasco_2023} we can consider the
smooth functional $\E \colon X_{+} \setminus \{0\} \to \mathbb{R}$
defined as
\begin{equation*}
	\E(w) = \J[P(w)](\gamma(w)) = \sup_{\eta \in B_{\lambda}} 
	\J[P(w)](\eta),
\end{equation*}
where, as in the proof of \ref{prop:massimo}, $P(w) =
\frac{w}{\norml{w}}$ and $\gamma(w) = \eta(Pw)$ and deduce that for
all $w \in \Sigma_{+}$ and $v \in X_{+}$.
\begin{equation}
	\label{eq:derivataE}
	d\E(w)[v] = a(\gamma(w))dI(\psi_{w})[v] - a(\gamma(w))^{2}
	\omega(\psi_{w}) \scalarl{w}{v}
\end{equation}
where
\begin{equation}
	\label{eq:moltiplicatore}
	\omega(\psi_{w}) = a(\gamma(w))^{-1} dI(\psi_{w})[w].
\end{equation}
We also deduce that, for all $h = v + \xi \in X$, $v \in X_{+}$, $\xi
\in X_{-}$ and with $\psi$, $a$ and $\omega$ as above we have
\begin{align*}
	dI(\psi)[h] - \omega \scalarl{\psi}{h} &= dI(\psi)[v] +
	dI(\psi)[\xi] -\omega \scalarl{a(\eta) w + \eta}{v + \xi} \\
	&= dI(\psi)[v] + \left(-da(\eta)[\xi] - \scalarl{w}{v} -
	\frac{\scalarl{\eta}{\xi}}{a(\eta)} \right) dI(\psi)[w] \\
	&= dI(\psi)[v - \scalarl{w}{v}w] = \frac{1}{a(\eta)} d\E(w)[v]
\end{align*}
which shows that $w$ is a critical point of $\E$ if and only if $\psi 
= a(\eta(w))w + \eta(w)$ is a critical point for $I$ under the 
constraint $\norml{\psi}^{2} = \lambda$.

The following is essentially Proposition 2.13 of
\cite{CotiZelati_Nolasco_2023}:
\begin{prop}
	\label{prop:punticriticiE}
	Let $w_{0} \in \Sigma_{+}$ be a critical point of $\E$
	restricted on the manifold $\Sigma_{+}$.
	Then $w_{0}$ is a critical point for $\E$ on $X_{+}$ and 
	the function
	\begin{equation*}
		\psi_{0} = a(\eta(w_{0})) w_{0} + \eta(w_{0}) \in \Sigma^{\lambda}
	\end{equation*}
	is a critical point for $I$ on the manifold $\Sigma$ and satisfies
	\begin{equation}
		\label{eq:equazionemoltiplicatore}
		dI(\psi_{0})[h] = \omega \scalarl{\psi_{0}}{h}
		\qquad \text{for all } h \in X
	\end{equation}
	where $\omega = \omega(\psi_{0}) \in \mathbb{R}$,
	\begin{equation*}
		\label{eq:stimeMoltiplicatore}
		(1 - \gamma C_{\alpha,\lambda}) \normh{w_{0}}^{2} \leq
		\omega(\psi_{0}) \leq 2 I(\psi_{w_{0}}) = 2\E(w_{0}).
	\end{equation*}
	
	Moreover, if $\psi_{0} \in \Sigma^{\lambda}$ satisfies
	\eqref{eq:equazionemoltiplicatore} for some $\omega > 0$, then $w
	= \norml{\Lambda_{+}\psi_{0}}^{-1} \Lambda_{+} \psi_{0}$ is a
	critical point for $\E(w)$ .
\end{prop}

\begin{proof}
	We only have to prove the estimate on the Lagrange multiplier. The
	other points follow as in proof of Proposition 2.13 of
	\cite{CotiZelati_Nolasco_2023}. Using \eqref{eq:stimaFw} we
	immediately deduce
	\begin{align*}
		\omega(\psi_{0}) &= a(\eta(w_{0}))^{-1} dI(\psi_{0})[w_{0}]
		\geq (1 - \gamma C_{\alpha,\lambda})\normh{w_{0}}^{2}.
	\end{align*}
	We also have, using \eqref{eq:ipotesiAR}, that
	\begin{equation*}
		\omega(\psi_{0}) = dI(\psi_{0})[\psi_{0}] \leq 2 I(\psi_{0}).
	\end{equation*}
\end{proof}

\section{The constrained minimization}

We introduce now the following minimization problem:
\begin{equation*}
	e(\lambda) = \inf_{w \in \Sigma_{+}} \E(w) = \inf_{w \in
	\Sigma_{+}} \left\{ \frac{1}{2} a(\eta(w))^{2}\normh{w}^{2} -
	\frac{1}{2} \normh{\eta(w)}^{2} - \gamma\int F(\psi_{w}) \right\}
\end{equation*}
where $\eta(w)$ is given in proposition \ref{prop:massimo}, $a(\eta(w))
= \sqrt{\lambda - \norml{\eta(w)}^{2}}$ and $\psi_{w} = a(\eta(w))w +
\eta(w)$.

The next lemma contains an estimate which will be essential in proving
(via the concentration compactness lemma \cite{Lions_1984,
Lions_1984-1}) convergence of minimizing sequences for $\E$.

\begin{lem}
	\label{lem:stimAe}
	For all $\lambda \in (0,1]$ we have that $0 < e(\lambda) <
	\frac{\lambda m}{2}$.
\end{lem}

\begin{proof}
	From lemma \ref{lem:stime_max} we have that 
	\begin{equation*}
		e(\lambda) \geq \frac{\lambda}{2 S_{2}^{2}} \left(1 - \gamma
		\left(S_{2}^{2} + 2\frac{\mu}{\alpha}
		\lambda^{\frac{\alpha-2}{2}}
		S^{2}_{\frac{4}{4-\alpha}}\right)\right) > 0.
	\end{equation*}
	
	Using lemma \ref{lem:stimautile} we deduce that
	\begin{align*}
		\E(w) &= I(\psi_{w}) = \frac{1}{2}a(\eta(w))^{2}\normh{w}^{2}
		- \frac{1}{2}\normh{\eta(w)}^{2} - \gamma \int F(\psi_{w}) \\
		&\leq \frac{\lambda}{2}\normh{w}^{2} -
		\frac{1}{2}\norml{\eta(w)}^{2}\normh{w}^{2} -
		\frac{1}{2}\normh{\eta(w)}^{2} - \gamma \int
		F(\sqrt{\lambda}w) \\
		&\qquad - \gamma \int \scalar{\nabla F(aw)}{\eta(w)} + \gamma
		\left(S_{2}^{2} + \mu \lambda^{\frac{\alpha-2}{2}}
		S^{2}_{\frac{4}{4-\alpha}}\right)
		\norml{\eta(w)}^{2} \normh{w}^{2} \\
		&\qquad + \gamma \left(S_{2}^{2} + \mu \lambda^{\frac{\alpha -
		2}{2}} S^{2}_{\frac{4}{4-\alpha}}\right) \normh{\eta(w)}^{2} +
		\gamma \mu \lambda^{\frac{\alpha - 2}{2}} \int
		\abs{w}^{\alpha-2} \abs{\eta(w)}^{2}
	\end{align*}
	Since $\gamma \left(S_{2}^{2} + \mu \lambda^{\frac{\alpha-2}{2}}
	S^{2}_{\frac{4}{4-\alpha}}\right) < \frac{1}{4}$ by
	\eqref{eq:ipotesigamma2} we have, for all $w \in \Sigma_{+}$:
	\begin{multline*}
		\E(w) \leq \frac{\lambda}{2}\normh{w}^{2} -
		\frac{1}{4}\norml{\eta(w)}^{2}\normh{w}^{2} -
		\frac{1}{4}\normh{\eta(w)}^{2} \\
		- \gamma \int F(\sqrt{\lambda}w) - \gamma \int \scalar{\nabla
		F(aw)}{\eta(w)} + \gamma \mu \lambda^{\frac{\alpha - 2}{2}} \int
		\abs{w}^{\alpha-2} \abs{\eta(w)}^{2}.
	\end{multline*}
	
	Fix $w_{1} \in H^{1}(\mathbb{R}^{3}, \mathbb{C}) \cap
	C^{1}(\mathbb{R}^{3}, \mathbb{C})$ such that $\norml{w_{1}} = 1$
	and let $w = \left(\begin{smallmatrix} w_{1} \\ 0
	\end{smallmatrix}\right) \in \mathbb{C}^{4}$ and $w_{\epsilon}(x)
	= \epsilon^{3/2} w(\epsilon x)$. Then $\norml{w_{\epsilon}} = 1$
	for all $\epsilon > 0$ and $\norml[\infty]{w_{\epsilon}(x)} =
	\epsilon^{3/2} \norml[\infty]{w(x)} \to 0$ as $\epsilon \to 0$. We
	also have that $\abs{\scalar{w(x)}{\beta w(x)}} =
	\abs{w_{1}(x)}^{2}$.
	
	The same computation of Lemma 2.16 in 
	\cite{CotiZelati_Nolasco_2023} show that
	\begin{align*}
		&\normh{w_{\epsilon}}^{2} - m\norml{w_{\epsilon}}^{2} \leq
		\frac{\epsilon^{2}}{2m} \int \abs{q}^{2}
		\abs{\hat{w}_{1}(q)}^{2}\\
		&\normh{w_{\epsilon}}^{2} \leq m + C\epsilon^{2} \\
		&\normh{w_{\epsilon} - \Lambda_{+}w_{\epsilon}}^{2} \leq
		\frac{\epsilon^{2}}{4m} \int \abs{p}^{2}
		\abs{\hat{w}_{1}(p)}^{2}\\
		&\labs{1 - \norml{\Lambda_{+}w_{\epsilon}}}^{2} \leq
		\frac{\epsilon^{2}}{4m} \int \abs{p}^{2}
		\abs{\hat{w}_{1}(p)}^{2}.
	\end{align*}
	We deduce from this that for $\epsilon > 0$ small enough 
	$\norml{\Lambda_{+}w_{\epsilon}} > \frac{1}{2}$.
	
	Let  
	\begin{equation*}
		\varphi_{\epsilon}(x) = \norml{\Lambda_{+} w_{\epsilon}}^{-1}
		\Lambda_{+} w_{\epsilon}(x) \in \Sigma_{+}.
	\end{equation*}
	We have that
	\begin{align*}
		&\normh{\varphi_{\epsilon}} \leq \norml{\Lambda_{+}
		w_{\epsilon}}^{-1} \normh{w_{\epsilon}} \leq \sqrt{m} + 
		C\epsilon\\
		&\normh{w_{\epsilon} - \varphi_{\epsilon}} \leq
		\frac{\epsilon}{2\sqrt{m}} (2 + \normh{\varphi_{\epsilon}})
		\left(\int \abs{p}^{2} \abs{\hat{w}_{1}(p)}^{2}\right)^{1/2} =
		\frac{\epsilon}{\sqrt{m}} (2 + \normh{\varphi_{\epsilon}})
		\norml{\nabla w_{1}}\\
		&\norml{\varphi_{\epsilon} - w_{\epsilon}} \leq
		\frac{\epsilon}{\sqrt{m}} \left(\int \abs{p}^{2}
		\abs{\hat{w}_{1}(p)}^{2}\right)^{1/2} =
		\frac{\epsilon}{\sqrt{m}} \norml{\nabla w_{1}}
	\end{align*}
	and hence, for $\epsilon$ small enough
	\begin{equation*}
		\normh{\varphi_{\epsilon}} \leq \sqrt{m} + 1, \qquad
		\normh{\eta(\varphi_{\epsilon})} \leq
		a(\eta(\varphi_{\epsilon})) \normh{\varphi_{\epsilon}} \leq
		\sqrt{\lambda}(\sqrt{m} + 1).
	\end{equation*}
	
	We now estimate the functional for small $\epsilon > 0$:
	\begin{multline}
		\label{eq:stime_funzionale}
		\E(\varphi_{\epsilon}) \leq \frac{\lambda m}{2} +
		\frac{\lambda}{2}(\normh{\varphi_{\epsilon}}^{2} - m
		\norml{\varphi_{\epsilon}}^{2}) -
		\frac{1}{4}\norml{\eta(\varphi_{\epsilon})}^{2}
		\normh{\varphi_{\epsilon}}^{2} -
		\frac{1}{4}\normh{\eta(\varphi_{\epsilon})}^{2} \\
		- \gamma \int F(\sqrt{\lambda}\varphi_{\epsilon}) - \gamma
		\int \scalar{\nabla F(a\varphi_{\epsilon})}
		{\eta(\varphi_{\epsilon})} + \gamma \mu
		(\sqrt{\lambda})^{\alpha - 2} \int
		\abs{\varphi_{\epsilon}}^{\alpha-2}
		\abs{\eta(\varphi_{\epsilon})}^{2}
	\end{multline}
	The first term can be estimated as follows:
	\begin{equation*}
		\normh{\varphi_{\epsilon}}^{2} -
		m\norml{\varphi_{\epsilon}}^{2} \leq \frac{\epsilon^{2}}{2m}
		(3 + \normh{\varphi_{\epsilon}})^{2} \norml{\nabla w_{1}}^{2} \leq \frac{\epsilon^{2}}{2m}
		(4 + \sqrt{m})^{2} \norml{\nabla w_{1}}^{2} \leq 
		c_{1}\epsilon^{2}.
	\end{equation*}
	
	Using \eqref{eq:ipotesiStimaFdalbasso} and
	\eqref{eq:ipotesistimegradbetter} and since
	$\norml[\infty]{w_{\epsilon}} \to 0$ we have that
	\begin{align*}
		\int F(\sqrt{\lambda}\varphi_{\epsilon}) &= \int
		F(\sqrt{\lambda}w_{\epsilon}) + \int \scalar {\nabla
		F(\sqrt{\lambda}w_{\epsilon})}{\varphi_{\epsilon} -
		w_{\epsilon}} \\
		&\qquad + \int \scalar {\nabla F(\sqrt{\lambda(}w_{\epsilon} +
		\theta(\varphi_{\epsilon} - w_{\epsilon}))) - \nabla
		F(\sqrt{\lambda}w_{\epsilon})}{\varphi_{\epsilon} -
		w_{\epsilon}} \\
		&\geq \bar{\gamma} \lambda^{\frac{\alpha}{2}} \int
		\abs{\scalar{w_{\epsilon}}{\beta
		w_{\epsilon}}}^{\frac{\alpha}{2}} - \lambda^{\frac{\nu}{2}}
		\norml[2\nu]{w_{\epsilon}}^{\nu} \norml{\varphi_{\epsilon} -
		w_{\epsilon}} \\
		&\qquad - \int \abs{\varphi_{\epsilon} - w_{\epsilon}}^{2} -
		\mu \int \abs{w_{\epsilon}}^{\alpha-2} \abs{\varphi_{\epsilon}
		- w_{\epsilon}}^{2} \\
		&\qquad - \mu \int ( \abs{\varphi_{\epsilon}} +
		\abs{w_{\epsilon}} )^{\alpha-2} \abs{\varphi_{\epsilon} -
		w_{\epsilon}}^{2} \\
		&\geq \bar{\gamma} \lambda^{\frac{\alpha}{2}} \int
		\abs{\scalar{w_{\epsilon}}{\beta
		w_{\epsilon}}}^{\frac{\alpha}{2}} - \lambda^{\frac{\nu}{2}}
		\norml[2\nu]{w_{\epsilon}}^{\nu} \norml{\varphi_{\epsilon} -
		w_{\epsilon}} \\
		&\qquad - \int \abs{\varphi_{\epsilon} - w_{\epsilon}}^{2} -
		2\mu \int \abs{w_{\epsilon}}^{\alpha-2} \abs{\varphi_{\epsilon}
		- w_{\epsilon}}^{2} \\
		&\qquad - 2\mu \int \abs{\varphi_{\epsilon}}^{\alpha-2}
		\abs{\varphi_{\epsilon} -
		w_{\epsilon}}^{2} \\
		&\geq \epsilon^{\frac{3(\alpha-2)}{2}} \bar{\gamma}
		\lambda^{\frac{\alpha}{2}} \int \abs{w_{1}}^{\alpha} -
		\lambda^{\frac{\nu}{2}} \epsilon^{\frac{3(\nu -
		1)}{2}}\norml[2\nu]{w_{1}}^{\nu} \norml{\varphi_{\epsilon} -
		w_{\epsilon}} \\
		&\qquad - \norml{\varphi_{\epsilon} - w_{\epsilon}}^{2} - 2\mu
		\left(\norml[\alpha]{w_{\epsilon}}^{\alpha-2} +
		\norml[\alpha]{\varphi_{\epsilon}}^{\alpha-2}\right)
		\norml[\alpha]{\varphi_{\epsilon} - w_{\epsilon}}^{2} \\
		&\geq \epsilon^{\frac{3(\alpha-2)}{2}} \bar{\gamma}
		\lambda^{\frac{\alpha}{2}} \int \abs{w_{1}}^{\alpha} - c_{2}
		\epsilon^{\frac{3\nu + 1}{2}} - c_{3}\epsilon^{2}.
 	\end{align*}

	For the next term we observe that 
	\begin{equation*}
		\int \scalar{\nabla F(a\varphi_{\epsilon})}
		{\eta(\varphi_{\epsilon})} \leq \int \abs{\nabla
		F(a\varphi_{\epsilon}) - \nabla F(aw_{\epsilon})}
		\abs{\eta(\varphi_{\epsilon})} + \int \abs{\nabla
		F(aw_{\epsilon})} \abs{\eta(\varphi_{\epsilon})}.
	\end{equation*}
	
	Let us analyze the first term in this expression: we have
	\begin{align*}
		&\int \abs{\nabla F(a\varphi_{\epsilon}) - \nabla
		F(aw_{\epsilon})} \abs{\eta(\varphi_{\epsilon})} \\
		&\quad \leq \int
		\left(\abs{a(\varphi_{\epsilon} - w_{\epsilon})} + \abs{aw_{\epsilon}} + \mu
		\abs{a(\varphi_{\epsilon} - w_{\epsilon})}^{\alpha-2} + \mu
		\abs{aw_{\epsilon}}^{\alpha-2} \right) a
		\abs{\varphi_{\epsilon} - w_{\epsilon}}
		\abs{\eta(\varphi_{\epsilon})} \\
		&\quad \leq a^{2} \int \left(\abs{(\varphi_{\epsilon} -
		w_{\epsilon})}^{2} + \abs{w_{\epsilon}}\abs{\varphi_{\epsilon}
		- w_{\epsilon}} \right) \abs{\eta(\varphi_{\epsilon})} \\
		&\qquad + \mu a^{\alpha-1} \int
		\left(\abs{\varphi_{\epsilon} - w_{\epsilon}}^{\alpha-1} +
		\abs{w_{\epsilon}}^{\alpha-2} \abs{\varphi_{\epsilon} -
		w_{\epsilon}}\right) \abs{\eta(\varphi_{\epsilon})} \\
		&\quad \leq a^{2} \norml[3]{\varphi_{\epsilon} -
		w_{\epsilon}}^{2} \norml[3]{\eta(\varphi_{\epsilon})} + a^{2}
		\norml[3]{w_{\epsilon}} \norml[3]{\varphi_{\epsilon} -
		w_{\epsilon}} \norml[3]{\eta(\varphi_{\epsilon})}\\
		&\qquad + \mu a^{\alpha-1}
		\norml[\alpha]{\varphi_{\epsilon} - w_{\epsilon}}^{\alpha-1}
		\norml[\alpha]{\eta(\varphi_{\epsilon})} + \mu a^{\alpha-1}
		\epsilon^{\frac{3(\alpha-2)}{2}}
		\norml[\infty]{w_{1}}^{\alpha-2}
		\norml{\varphi_{\epsilon} - w_{\epsilon}}
		\norml{\eta(\varphi_{\epsilon})} \\
		&\quad \leq a^{2} \norml[3]{\eta(\varphi_{\epsilon})} \left(
		S_{3}^{2} \normh{\varphi_{\epsilon} - w_{\epsilon}}^{2} +
		\norml[3]{w_{\epsilon}} S_{3} \normh{\varphi_{\epsilon} -
		w_{\epsilon}} \right)\\
		&\qquad + \mu a^{\alpha-1} S_{\alpha}^{\alpha}
		\normh{\eta(\varphi_{\epsilon})} \normh{\varphi_{\epsilon} -
		w_{\epsilon}}^{\alpha-1} 
		\\
		&\qquad 
		+ \mu a^{\alpha-1}
		\epsilon^{\frac{3(\alpha-2)}{2}}
		\norml[\infty]{w_{1}}^{\alpha-2}
		\norml{\varphi_{\epsilon} - w_{\epsilon}}
		\norml{\eta(\varphi_{\epsilon})} 
		\\
		&\quad \leq \lambda^{\frac{3}{2}} S_{3}^{2} (\sqrt{m} + 1)
		\left( \epsilon^{2} S_{3} \frac{(3 + \sqrt{m})^{2}}{m}
		\norml{\nabla w_{1}}^{2} + \epsilon^{\frac{3}{2}}
		\norml[3]{w_{1}} \frac{3+\sqrt{m}}{\sqrt{m}} \norml{\nabla
		w_{1}} \right)\\
		&\qquad + \epsilon^{\alpha-1} \mu
		\lambda^{\frac{2\alpha-1}{2}} S_{\alpha}^{\alpha} (\sqrt{m} + 1)
		\left(\frac{3+\sqrt{m}}{\sqrt{m}}\right)^{\alpha-1}
		\norml{\nabla w}^{\alpha-1} \\
		&\qquad + \epsilon^{\frac{3\alpha-4}{2}} 
		\mu \lambda^{\frac{2\alpha-1}{2}}
		S_{2}^{2} \norml[\infty]{w_{1}}^{\alpha-2} 
		\frac{(\sqrt{m} + 1)}{\sqrt{m}} \norml{\nabla w_{1}} \\
		&\quad \leq c_{4}\epsilon^{2} + c_{5}\epsilon^{\frac{3}{2}} +
		c_{6}\epsilon^{\alpha-1} + c_{7}
		\epsilon^{\frac{3\alpha-4}{2}}
	\end{align*}
	while for the second term we have
	\begin{multline*}
		\int \abs{\nabla F(aw_{\epsilon})}
		\abs{\eta(\varphi_{\epsilon})} \leq a^{\nu} \int
		\abs{w_{\epsilon}}^{\nu} \abs{\eta(\varphi_{\epsilon})}
		\leq \frac{1}{2}a^{2\nu} \norml[2\nu]{w_{\epsilon}}^{2\nu} +
		\frac{1}{2} \norml{\eta(\varphi_{\epsilon})}^{2}  \\
		\leq \epsilon^{3(\nu-1)} \frac{\lambda^{\nu}}{2}
		\norml[2\nu]{w_{1}}^{2\nu} + \frac{S_{2}^{2}}{2}
		\normh{\eta(\varphi_{\epsilon})}^{2} \leq c_{8}
		\epsilon^{3(\nu-1)} + \frac{S_{2}^{2}}{2}
		\normh{\eta(\varphi_{\epsilon})}^{2}
	\end{multline*}
	The last term in \eqref{eq:stime_funzionale} can be estimated as 
	follows:
	\begin{align*}
		\int \abs{\varphi_{\epsilon}}^{\alpha-2}
		\abs{\eta(\varphi_{\epsilon})}^{2} &\leq
		\epsilon^{\frac{3(\alpha-2)}{2}}
		\norml[\infty]{\varphi_{1}}^{\alpha - 2}
		\norml{\eta(\varphi_{\epsilon})}^{2} \leq
		\epsilon^{\frac{3(\alpha-2)}{2}}
		\norml[\infty]{\varphi_{1}}^{\alpha-2} S_{2}^{2}
		\normh{\eta(\varphi_{\epsilon})}^{2}
	\end{align*}
	We therefore deduce that
	\begin{align*}
		\E(\varphi_{\epsilon}) &\leq \frac{\lambda m}{2} + c_{1}
		\epsilon^{2} - \frac{1}{4}\norml{\eta(\varphi_{\epsilon})}^{2}
		\normh{\varphi_{\epsilon}}^{2} -
		\frac{1}{4}\normh{\eta(\varphi_{\epsilon})}^{2}
		-\epsilon^{\frac{3(\alpha-2)}{2}} \gamma\bar{\gamma}
		\lambda^{\frac{\alpha}{2}} \int \abs{w_{1}}^{\alpha} \\
		&\qquad + c_{2} \gamma \epsilon^{\frac{3\nu + 1}{2}} + c_{3}
		\gamma \epsilon^{2} + c_{4} \gamma\epsilon^{2} + c_{5}
		\gamma\epsilon^{\frac{3}{2}} + c_{6}\gamma\epsilon^{\alpha-1}
		+ c_{7}\gamma \epsilon^{\frac{3\alpha-4}{2}} \\
		&\qquad + c_{8} \gamma \epsilon^{3(\nu-1)} + \gamma
		\frac{S_{2}^{2}}{2} \normh{\eta(\varphi_{\epsilon})}^{2} +
		\epsilon^{\frac{3(\alpha-2)}{2}}
		\gamma\norml[\infty]{\varphi_{1}}^{\alpha-2} S_{2}^{2}
		\normh{\eta(\varphi_{\epsilon})}^{2}
	\end{align*}
	
	Since $\gamma S_{2}^{2} < \frac{1}{4}$ by \eqref{eq:ipotesigamma2}
	and
	\begin{equation*}
		\frac{3(\alpha-2)}{2} < \min\left\{2, \frac{3\nu + 1}{2},
		\frac{3}{2}, \alpha-1, \frac{3\alpha-4}{2}, 3(\nu-1) \right\}
		\qquad \text{for all } \alpha \in (2, \frac{8}{3}]
	\end{equation*}
	we deduce that 
	\begin{equation*}
		\E(\varphi_{\epsilon}) < \frac{\lambda m}{2}
	\end{equation*}
	for $\epsilon$ small enough and for all $\lambda \in (0,1]$ and
	hence $e(\lambda) < \frac{\lambda m}{2}$.
\end{proof}

\begin{prop}
	\label{prop:stimee}
	For all $\lambda \in (0,1)$ and $\theta > 1$ such that
	$\theta\lambda < 1$ we have that
	\begin{equation*}
		\label{eq:stimee}
		e(\theta \lambda)  < \theta e(\lambda).
	\end{equation*}
\end{prop}

\begin{proof}
	Let $\theta > 1$ be such that $\lambda \theta \in (0,1)$.
	Take $w \in \Sigma_{+}$ and let $\etl(w) \in B_{\theta
	\lambda}$ be the function whose existence follows from proposition
	\ref{prop:massimo}. Let $\pel(w) = a_{\theta\lambda}(\etl(w))w + 
	\etl(w)
	\in \Sigma_{\theta\lambda}$. We have that
	\begin{equation*}
		\mathcal{E}_{\theta\lambda}(w) = \frac{1}{2}
		\normh{a_{\theta\lambda}(\etl(w))w}^{2} - \frac{1}{2}
		\normh{\etl(w)}^{2} - \gamma \int F(\pel(w)).
	\end{equation*}
	
	We recall that from assumption \eqref{eq:ipotesiAR} follows that
	\begin{equation*}
		F(\sigma x) \geq \sigma^{\alpha} F(x) \geq \sigma^{2} F(x) + 
		\left(\sigma^{\alpha} - \sigma^{2}\right) F(x) \qquad 
		\text{for all } x \in \mathbb{C}^{4}, \sigma \geq 1.
	\end{equation*}

	We have that for all $\eta \in B_{\theta
	\lambda}$ 
	\begin{equation*}
		a_{\lambda \theta}(\eta) = \sqrt{\lambda \theta -
		\norml{\eta}^{2}} = \sqrt{\theta} \sqrt{\lambda -
		\norml{\frac{\eta}{\sqrt{\theta}}}^{2}} = \sqrt{\theta}
		a_{\lambda}(\frac{\eta}{\sqrt{\theta}})
	\end{equation*}
	and hence we have that
	\begin{align*}
		e(\theta\lambda) &\leq \mathcal{E}_{\theta\lambda}(w) = \theta
		\left[\frac{1}{2} \normh{a_{\lambda}
		(\frac{\etl(w)}{\sqrt{\theta}})w}^{2} - \frac{1}{2}
		\normh{\frac{\etl(w)}{\sqrt{\theta}}}^{2} -
		\frac{\gamma}{\theta}\int F\left(\sqrt{\theta}
		\frac{\pel(w)}{\sqrt{\theta}} \right)\right] \\
		&\leq \theta \left[\frac{1}{2}
		\normh{a_{\lambda}(\frac{\etl(w)}{\sqrt{\theta}})w}^{2} -
		\frac{1}{2} \normh{\frac{\etl(w)}{\sqrt{\theta}}}^{2} - \gamma
		\int F\left( \frac{\pel(w)}{\sqrt{\theta}}\right) \right] \\
		&\qquad - \gamma\theta\left(\theta^{(\alpha-2)/2} - 1\right) \int
		F\left( \frac{\pel(w)}{\sqrt{\theta}} \right) \\
		&\leq \theta \mathcal{E}_{\lambda}(w) - \gamma
		\theta\left(\theta^{(\alpha-2)/2} - 1\right) \int F\left(
		\frac{\pel(w)}{\sqrt{\theta}} \right) .
	\end{align*}

	We claim that for all $w \in \Sigma_{+}$ such that $\E(w) 
	< \frac{1}{2} (\frac{\lambda m}{2} + e(\lambda))$
	\begin{equation*}
		\int F\left( \frac{\pel(w)}{\sqrt{\theta}} \right) \geq \delta
		> 0.
	\end{equation*}
	Assuming the claim, let $w_{n} \in \Sigma_{+}$ be a sequence which
	minimize $\E(w)$, $\E(w_{n}) \to e(\lambda) < \frac{1}{2}
	(\frac{\lambda m}{2} + e(\lambda))$. We have that for all $\theta
	> 1$
	\begin{multline*}
		e(\theta\lambda) \leq \mathcal{E}_{\theta\lambda}(w_{n}) \leq
		\theta \mathcal{E}_{\lambda}(w_{n}) -
		\theta\left(\theta^{(\alpha-2)/2} - 1\right)\delta \\
		= \theta e(\lambda) - \theta\left(\theta^{(\alpha-2)/2} -
		1\right)\delta + o(1) < \theta e(\lambda)
	\end{multline*}
	and the proposition follows.

	To prove the claim, we assume it does not hold and that there is a
	sequence $w_{n} \in \Sigma_{+}$ such that $\E(w_{n}) <
	\frac{1}{2} (\frac{\lambda m}{2} + e(\lambda))$ and
	\begin{equation*}
		\int F\left( \frac{\pel(w_{n})}{\sqrt{\theta}} \right) \to 0.
	\end{equation*}
	Then also $\int F(\pel(w_{n})) \to 0$ and we deduce from
	\eqref{eq:ipotesiES} that for all $\zeta > 0$
	\begin{multline*}
		\labs{\int \scalar{\nabla F(\pel(w_{n}))}{w_{n}}} \leq \int
		(\zeta + C_{\zeta} F(\pel(w_{n}))^{\frac{1}{\xi}} ) \abs{\pel(w_{n})}
		\abs{w_{n}} \\
		\leq \zeta \norml{\pel(w_{n})} \norml{w_{n}} + C_{\zeta}
		\norml[\frac{2\xi}{\xi-1}]{\pel(w_{n})}
		\norml[\frac{2\xi}{\xi-1}]{w_{n}} \int F(\pel(w_{n})) \leq 
		C\zeta + o(1)
	\end{multline*}
	which implies that 
	\begin{equation*}
		\lim_{n \to +\infty} \labs{\int \scalar{\nabla
		F(\pel(w_{n}))}{w_{n}}} = 0
	\end{equation*}
	and similarly
	\begin{equation*}
		\lim_{n \to +\infty} \labs{\int \scalar{\nabla
		F(\pel(w_{n}))}{\etl(w_{n})}} = 0
	\end{equation*}
	Since 
	\begin{multline*}
		0 = -\norml{\etl(w_{n})}^{2}\normh{w_{n}}^{2} -
		\normh{\etl(w_{n})}^{2} \\
		- \gamma \int \scalar{\nabla F(\pel(w_{n}))}{-
		\frac{\norml{\etl(w_{n})}^{2}}
		{a_{\theta\lambda}(\etl(w_{n}))} w_{n} + \etl(w_{n})}
	\end{multline*}
	we deduce from the boundedness of $\normh{w_{n}}$ and
	$\normh{\etl(w_{n})}$ that $\normh{\etl(w_{n})} \to 0$ and hence
	\begin{align*}
		\frac{\theta}{2} (\frac{\lambda m}{2} &+ e(\lambda)) >
		\theta\E(w_{n}) \geq \E[\theta](w_{n}) \\
		&= \frac{\theta\lambda}{2}\normh{w_{n}}^{2}
		-\frac{1}{2} \norml{\etl(w_{n})}^{2}\normh{w_{n}}^{2} -
		\frac{1}{2} \normh{\etl(w_{n})}^{2} - \gamma \int F(\pel(w_{n}))\\
		&\geq \frac{\theta\lambda m}{2} + o(1),
	\end{align*}
	a contradiction which proves the claim.
\end{proof}

\begin{lem}
	\label{lem:stimedalbasso}
	Let $\omega \in (0,m)$, $C > 0$ and $\psi \in X$ be such that
	\begin{equation}
		\label{eq:critical}
		\begin{cases}
			dI(\psi)[h] = \omega \scalarl{\psi}{h} \qquad \text{for
			all } h \in X \\
			\normh{\psi} \leq C
		\end{cases}
	\end{equation}
	Then there exists two positive constants, $c_{1}$ and $c_{2}$,
	which depend only on $\omega$ and $C$, such that
	\begin{equation*}
		\normh{\psi} \geq c_{1} > 0, \qquad I(\psi) \geq c_{2} > 0.
	\end{equation*}
\end{lem}

\begin{proof}
	Taking $h = \Lambda_{+}\psi - \Lambda_{-}\psi$ in
	\eqref{eq:critical} and $\epsilon > 0$ we have
	\begin{align*}
		\normh{\Lambda_{+}\psi}^{2} &+ \normh{\Lambda_{-}\psi}^{2} + 
		\omega\norml{\Lambda_{-}\psi}^{2} = 
		\omega\norml{\Lambda_{+}\psi}^{2} + \int \scalar{\nabla 
		F(\psi)}{\Lambda_{+}\psi - \Lambda_{-}\psi} \\
		&\leq \frac{\omega}{m}\normh{\Lambda_{+}\psi}^{2} + \int
		(\epsilon\abs{\psi} + \mu_{\epsilon}\abs{\psi}^{\alpha-1}
		)(\abs{\Lambda_{+}\psi} + \abs{\Lambda_{-}\psi}) \\
		&\leq \frac{\omega}{m}\normh{\Lambda_{+}\psi}^{2} +
		2\epsilon\norml{\psi}^{2} +
		\mu_{\epsilon}\norml[\alpha]{\psi}^{\alpha-1}
		(\norml[\alpha]{\Lambda_{+}\psi} +
		\norml[\alpha]{\Lambda_{-}\psi}) \\
		&\leq \frac{\omega}{m}\normh{\Lambda_{+}\psi}^{2} + 2\epsilon
		C_{2}^{2}\normh{\psi}^{2} +
		\mu_{\epsilon}C_{\alpha}^{\alpha}\normh{\psi}^{\alpha}.
	\end{align*} 
	If $\epsilon > 0$ is such that $2\epsilon C_{2}^{2} < \frac{1 -
	\frac{\omega}{m}}{2}$ we deduce that
	\begin{equation*}
		\frac{1}{2}\left(1 - \frac{\omega}{m} \right) \normh{\psi}^{2}
		\leq \mu_{\epsilon}C_{\alpha}^{\alpha}\normh{\psi}^{\alpha}
	\end{equation*}
	and 
	\begin{equation*}
		\normh{\psi}^{\alpha-2} \geq
		\frac{1}{2\mu_{\epsilon}C_{\alpha}^{\alpha}} \left(1 -
		\frac{\omega}{m} \right).
	\end{equation*}
	Let us now estimate the critical level:
	\begin{align*}
		I(\psi) &= I(\psi) - \frac{1}{2} dI(\psi)[\psi] + \frac{1}{2} 
		\omega \norml{\psi}^{2}\\
		&= \int \left(\frac{1}{2} \scalar{\nabla F(\psi)}{\psi} -
		F(\psi)\right) + \frac{1}{2} \omega \norml{\psi}^{2} \\
		&\geq \left(\frac{\alpha}{2} - 1\right) \int F(\psi) +
		\frac{1}{2} \omega \norml{\psi}^{2}.
	\end{align*}
	We claim that there is $\delta > 0$ such that $\int F(\psi) \geq
	\delta > 0$ for all $\psi$ for which \eqref{eq:critical} holds. If
	not, we find a sequence $\psi_{n}$ of solutions such that $\int
	F(\psi_{n}) \to 0$. From
	\begin{equation*}
		0 = dI(\psi_{n})[\Lambda_{\pm}\psi_{n}] - \omega
		\norml{\Lambda_{\pm}\psi_{n}}^{2}
	\end{equation*}
	we deduce that  
	\begin{align*}
		&\normh{\Lambda_{+}\psi_{n}}^{2} - \omega \norml{\Lambda_{+}
		\psi_{n}}^{2} = \int \scalar{\nabla F
		(\psi_{n})}{\Lambda_{-}\psi_{n}} \\
		&\normh{\Lambda_{-}\psi_{n}}^{2} + \omega \norml{\Lambda_{-} \psi_{n}}^{2} = 
		-\int \scalar{\nabla F (\psi_{n})}{\Lambda_{-}\psi_{n}}.
	\end{align*}
	As in the proof of \ref{prop:stimee}, using \eqref{eq:ipotesiES} 
	we can deduce from $\int F(\psi_{n}) \to 0$ that 
	\begin{equation*}
		\int \scalar{\nabla F (\psi_{n})}{\Lambda_{\pm}\psi_{n}} \to 0
	\end{equation*}
	and hence also 
	\begin{equation*}
		\normh{\psi_{n}} \to 0,
	\end{equation*}
	contradiction which proves the claim.
\end{proof}

\begin{proof}[Proof of theorem]

By Ekeland's variational principle, there exists a sequence $w_{n} \in
\Sigma_{+}$ such that 
\begin{equation*}
	\Eu(w_{n}) \to e(1), \qquad \sup_{v \in \Sigma_{+}}
	\labs{d\Eu(w_{n})[v]} \to 0.
\end{equation*}
Since $\Eu(w_{n}) \to e(1)$ we deduce from lemma \ref{lem:stime_max}
that the sequence $w_{n}$ is bounded. Follows from proposition
\ref{lem:stime_grad} that also $\eta_{n} = \eta(w_{n})$ and $\psi_{n}
= a(\eta_{n}) w_{n} + \eta_{n}$ are bounded in $X$. Letting
$\omega_{n} = a(\eta_{n})^{-1}dI(\psi_{n})[w_{n}] \geq
(1-C_{\alpha,1})$ we have that
\begin{equation*}
	dI(\psi_{n})[h] - \omega_{n} \scalarl{\psi_{n}}{h} =  
	\frac{1}{a(\eta_{n})} d\E(w_{n})[v] \to 0 \qquad
	\text{for all } h = v + \xi \in X.
\end{equation*}

We can assume that (up to a subsequence) $\psi_{n} \tow \psi$ in $X$
and that $\omega_{n} \to \omega$, with $\omega \in (1-C_{\alpha,1},
2\Eu(w_{n})) \subset (1-C_{\alpha,1}, m)$ for $n$ sufficiently large.
Then we have that for all $h \in X$
\begin{multline*}
	dI(\psi_{n})[h] - \omega_{n} \scalarl{\psi_{n}}{h} =
	\scalarh{\psi_{n}}{\Lambda_{+} h} - \scalarh{\psi_{n}}{\Lambda_{-}
	h} \\
	- \gamma \int \scalar{\nabla F(\psi_{n})}{h} - \omega_{n}
	\scalarl{\psi_{n}}{h} \to 0
\end{multline*}
since we have that
\begin{equation*}
	\int \scalar{\nabla F(\psi_{n}) - \nabla F(\psi)}{h} \to 0
\end{equation*}
we deduce that
\begin{equation*}
	dI(\psi)[h] - \omega \scalarl{\psi}{h} = 0 \qquad \text{for all }
	h \in X.
\end{equation*}
The weak convergence does not, however, preserve the $L^{2}$ norm, so 
we only know that $\norml{\psi} \leq \norml{\psi_{n}} = 1$ (it could 
even be that $\psi = 0$). 

To conclude we will now apply the concentration-compactness principle,
see \cite{Lions_1984, Lions_1984-1}. First of all let us show that no
vanishing occurs. By contradiction, assume that
\begin{equation*}
	\limsup_{n \to +\infty} \sup_{y \in \mathbb{R}^{3}} \int_{B(y,1)} 
	\abs{\psi_{n}}^{2} = 0.
\end{equation*}
Then we know, see \cite{Lions_1984} or \cite[Lemma 1.21]{Willem_1996},
that $\psi_{n} \to 0$ in $L^{p}(\mathbb{R}^{3})$ for $2 < p < 3$.
Then for any $\epsilon > 0$
\begin{equation*}
	\labs{\int \scalar{\nabla F(\psi_{n})}{w_{n}}} \leq \epsilon \int
	\abs{\psi_{n}} \abs{w_{n}} + \mu_{\epsilon} \int
	\abs{\psi_{n}}^{\alpha-1} \abs{w_{n}} \leq \epsilon +
	\mu_{\epsilon} \norml[\alpha]{\psi_{n}}^{\alpha-1}
	\norml[\alpha]{w_{n}}
\end{equation*}
and
\begin{equation*}
	\lim_{n \to +\infty} \labs{\int \scalar{\nabla 
	F(\psi_{n})}{w_{n}}} = 0.
\end{equation*}
We deduce from
\begin{equation*}
	\epsilon_{n} = dI(\psi_{n})[w_{n}] - \omega_{n}
	a(\eta_{n})\norml{w_{n}}^{2} = a(\eta_{n})(\normh{w_{n}}^{2} - 
	\omega_{n} \norml{w_{n}}^{2}) -\int \scalar{\nabla
	F(\psi_{n})}{w_{n}} 
\end{equation*}
that for all $\lambda \in (0,1]$
\begin{equation*}
	\epsilon_{n} = a(\eta_{n})(\normh{w_{n}}^{2} - \omega_{n}
	\norml{w_{n}}^{2}) \geq a(\eta_{n})(m - \omega_{n})
	\norml{w_{n}}^{2} \geq \sqrt{\frac{1}{2}}( m - \omega_{n}) > 0
\end{equation*}
a contradiction since, by proposition \ref{prop:punticriticiE} and
lemma \ref{lem:stimAe} we have that $\omega_{n} \leq 2\Eu(w_{n})$
and $e(1) < \frac{m}{2}$. Hence vanishing does not occur.

Since $\normh{\psi} \leq \normh{\psi_{n}}$ and $I(\psi) \geq c_{2}$
(by lemma \ref{lem:stimedalbasso}) we deduce from the concentration
compactness principle that there exists $p \geq 1$ functions
$\phi_{1}$, \ldots, $\phi_{p} \in X$, critical points for $I$ under
the constraint $\norml{\psi}^{2} = \mu_{i} \in (0,1]$ (satisfying
\eqref{eq:equazionemoltiplicatore} with $\omega = \lim_{n}\omega_{n} >
0$) and $p$ sequences of points $x_{i, n} \in \mathbb{R}^{3}$, $i =
1$, \ldots, $p$ such that $\abs{x_{i,n} - x_{j,n}} \to +\infty$ for
all $i \neq j$ as $n \to +\infty$ and
\begin{equation*}
	\normh{\psi_{n} - \sum_{i = i}^{^{p}} \phi_{i}(\cdot - x_{i,n})} 
	\to 0 \quad \text{as } n \to +\infty.
\end{equation*}
From this follows also that $\norml{\psi_{n}}^{2} = 1 = \sum_{i =
1}^{p} \mu_{i}$.

We then observe that
\begin{align*}
	\normh{\Lambda_{+}\psi_{n}}^{2} - &\normh{\Lambda_{-}\psi_{n}}^{2}
	= \scalarh{\psi_{n}}{\Lambda_{+}\psi_{n} - \Lambda_{-}\psi_{n}}
	\\
	&=\scalarh{\psi_{n} - \sum_{i = i}^{^{p}} \phi_{i}(\cdot -
	x_{i,n})}{\Lambda_{+}\psi_{n} - \Lambda_{-}\psi_{n}} \\
	&\qquad + \sum_{i = i}^{^{p}} \scalarh{\phi_{i}(\cdot -
	x_{i,n})}{\Lambda_{+}\psi_{n} - \Lambda_{-}\psi_{n}} \\
	&= \sum_{i = i}^{^{p}} \left(\scalarh{\Lambda_{+}\phi_{i}(\cdot -
	x_{i,n})}{\psi_{n}} - \scalarh{\Lambda_{-}\phi_{i}(\cdot -
	x_{i,n})}{\psi_{n}}\right) + o(1) \\
	&= \sum_{i = i}^{^{p}} \left(\normh{\Lambda_{+}\phi_{i}}^{2} -
	\normh{\Lambda_{-}\phi_{i}}^{2} \right) + o(1)
\end{align*}
and also
\begin{equation*}
	\int F(\psi_{n}) = \sum_{i=1}^{p} \int F(\phi_{i}(x)) + o(1).
\end{equation*}
(this can be deduced from the Brezis-Lieb lemma, see \cite[Theorem
2]{Brezis_Lieb_1983}). Finally we have that
\begin{equation}
	\label{eq:Isispezza}
	e(1) = I(\psi_{n})+ o(1) = \sum_{i=1}^{p} I(\phi_{i}) + o(1)
\end{equation}
and
\begin{equation*}
	0 = dI(\phi_{i})[h] - \omega \scalarl{\phi_{i}}{h} \qquad
	\text{for all } h \in X,
\end{equation*}
with $\omega > 0$. Follows from proposition \ref{prop:punticriticiE}
that $w_{i} = \norml{\Lambda_{+}\phi_{i}}^{-1} \Lambda_{+}\phi_{i} \in
\Sigma_{+}$ is a critical point for $\mathcal{E}_{\mu_{i}}$ and
$\mathcal{E}_{\mu_{i}}(w_{i}) = I(\phi_{i})$.

Since
\begin{equation*}
	\mathcal{E}_{\mu_{i}}(w_{i}) \geq e(\mu_{i}) 
\end{equation*}
we deduce from \eqref{eq:stimee} that if $p > 1$ (i.e. if dichotomy 
occurs)
\begin{multline*}
	e(1) = \sum_{i=1}^{p} I(\phi_{i}) + o(1) = \sum_{i=1}^{p}
	\mathcal{E}_{\mu_{i}}(w_{i}) + o(1) \geq \sum_{i=1}^{p} 
	e(\mu_{i}) + o(1)
	\\
	> \sum_{i=1}^{p} \mu_{i} e(\frac{1}{\mu_{i}}\mu_{i}) + o(1) =
	e(1)\sum_{i = 1}^{p} \mu_{i} + o(1),
\end{multline*}
a contradiction.
	
Since there is no vanishing or dichotomy, our sequence $\psi_{n}$
converges -- up to a translation -- strongly in $X$ to a critical
point $\psi \in X$ of \eqref{eq:equazionemoltiplicatore} such that
$\norml{\psi} = 1$ and the theorem follows.

\end{proof}

\subsection*{Acknowledgment} The authors sincerely thank the referee
for several useful comments and in particular for pointing out a
mistake in a first version of the manuscript.

\providecommand{\noopsort}[1]{} \providecommand{\cprime}{$'$}
  \providecommand{\scr}{} \def\ocirc#1{\ifmmode\setbox0=\hbox{$#1$}\dimen0=\ht0
  \advance\dimen0 by1pt\rlap{\hbox to\wd0{\hss\raise\dimen0
  \hbox{\hskip.2em$\scriptscriptstyle\circ$}\hss}}#1\else {\accent"17 #1}\fi}

\end{document}